\documentclass[11pt,oneside,leqno]{amsart}
\usepackage{amsxtra}
\usepackage{amsopn}
\usepackage{color}
\usepackage{amsmath,amsthm,amssymb}
\usepackage{amscd}
\usepackage{amsfonts}
\usepackage{latexsym}
\usepackage{verbatim}
\usepackage{multirow}
\usepackage{graphicx}

\theoremstyle{plain}
\newtheorem{theorem}{Theorem}[section]
\newtheorem{definition}[theorem]{Definition}
\newtheorem{lemma}[theorem]{Lemma}
\newtheorem{proposition}[theorem]{Proposition}

\newtheorem{remark}[theorem]{Remark}

\newtheorem{remark-question}[section]{Remark-Question}

\newcommand\C{{\mathbb C}}

\newcommand\Q{{\mathbb Q}}
\newcommand\R{{\mathbb R}}

\newcommand\fra{{\mathfrak a}} %%%%% Lie algebras %%%%%%
\newcommand\frg{{\mathfrak g}}
\newcommand\frh{{\mathfrak h}}
\newcommand\frk{{\mathfrak K}}

\newcommand\frs{{\mathfrak s}}
\newcommand\frt{{\mathfrak t}}
\newcommand\gc{\frg_{_\mathbb{C}}}
\newcommand\tc{\frt_{_\mathbb{C}}}
\newcommand\Real{{\mathfrak R}{\frak e}\,} %%%%% Real and Imaginary parts %%%%%
\newcommand\Imag{{\mathfrak I}{\frak m}\,}
\newcommand\nilm{\Gamma\backslash G}

\newcommand\db{{\bar{\partial}}}

\date{\today}

\begin{document}

\title[]{
%%%\small{Fr\"olicher spectral sequence of complex nilmanifolds\\ with balanced metrics}\\
%%%\small{On the Fr\"olicher spectral sequence of balanced Hermitian manifolds}\\
%%%\small{Fr\"olicher sequence of balanced nilmanifolds}
%%%\small{On the Fr\"olicher spectral sequence of compact balanced or pluriclosed  manifolds}\\
\small{Fr\"olicher spectral sequence of compact complex manifolds with special Hermitian metrics}\\
}

%%%\thanks{This work was supported by the Project...}

\author{Adela Latorre}
\address[A. Latorre]{Departamento de Matem\'atica Aplicada,
Universidad Polit\'ecnica de Madrid,
Avda. Juan de Herrera 4,
28040 Madrid, Spain}
\email{adela.latorre@upm.es}

\author{Luis Ugarte}
\address[L. Ugarte]{Departamento de Matem\'aticas\,-\,I.U.M.A.\\
Universidad de Zaragoza\\
Campus Plaza San Francisco\\
50009 Zaragoza, Spain}
\email{ugarte@unizar.es}

\author{Raquel Villacampa}
\address[R. Villacampa]{Departamento de Matem\'aticas\,-\,I.U.M.A.\\
Universidad de Zaragoza\\
Campus Plaza San Francisco\\
50009 Zaragoza, Spain}
\email{raquelvg@unizar.es}

%%%\date{\today}

\maketitle

\begin{abstract}
In this paper we focus on the interplay between the behaviour of the Fr\"olicher spectral sequence and the existence of special Hermitian metrics on the manifold, such as balanced, SKT or generalized Gauduchon. The study of balanced metrics on nilmanifolds endowed with strongly non-nilpotent complex structures allows us to provide 
infinite families of compact balanced manifolds with Fr\"olicher spectral sequence not degenerating at the second page. 
Moreover, this result is extended to non-degeneration at any arbitrary page. 
Similar results are obtained for the Fr\"olicher spectral sequence of compact generalized Gauduchon manifolds. 
We also find a compact SKT manifold whose  
Fr\"olicher spectral sequence does not degenerate at the second page, 
thus providing a counterexample to a conjecture by Popovici.
\end{abstract}

\section{Introduction}\label{introduction}

\noindent
Let $X$ be a complex manifold. 
%%%of complex dimension $n$. 
Fr\"olicher introduced in \cite{Fro} a spectral sequence $\{E^{*,*}_{r}(X)\}_{r\geq 1}$ associated to the double complex $(\Omega^{*,*}(X),\partial,\db)$, where $\partial+\db=d$. We will refer to it as the Fr\"olicher spectral sequence (FSS for short) of $X$. This sequence is also 
known as the Hodge-de Rham spectral sequence, as its first page is given by the Dolbeault cohomology $H^{*,*}_{\db}(X)$ of $X$ and it converges to the de Rham cohomology $H^*_{dR}(X,\C)$. Hence, 
%%%when the sequence does not degenerate at the first page, 
the spaces $E^{*,*}_{r}(X)$ provide (possibly new) complex invariants of the manifold. 

An important question is to understand the interplay between the behaviour of the Fr\"olicher spectral sequence and the existence of special Hermitian metrics on the manifold. It is well-known that the FSS of any compact K\"ahler manifold degenerates at the first page. In the context of non-K\"ahler Hermitian geometry,
well-known relevant classes of metrics arise, as for instance  \emph{balanced} or \emph{generalized Gauduchon} (in particular, \emph{pluriclosed}) metrics.
%{\color{blue} 
%one may ask if similar results hold for other relevant classes of metrics, such as \emph{balanced}, \emph{pluriclosed} or, more generally, \emph{generalized Gauduchon} metrics}. 
The aim of this paper is to construct compact complex manifolds endowed with these types of metrics and having Fr\"olicher spectral sequence not degenerating at different pages.

Recall that a Hermitian metric $F$ on a compact complex manifold $X$ with $\dim_{\C}X=n$ is called balanced if the form $F^{n-1}$ is closed \cite{Mi}.
The Iwasawa manifold is an example of a compact balanced manifold with FSS degenerating at the second page, with $E_1\not= E_2$. More generally, any compact quotient $X$ of a nilpotent \emph{complex} Lie group $G$ by a lattice is balanced (indeed, any left-invariant Hermitian metric on $G$ is balanced) and satisfies $E_2(X) = E_\infty(X)$ \cite{Sakane}. 

The Iwasawa manifold, as well as the compact quotients of nilpotent complex Lie groups, are special concrete examples of \emph{nilmanifolds}. 
There are some balanced nilmanifolds in the literature with FSS satisfying $E_2\neq E_{\infty}$. For instance, Cordero, Fern\'andez and Gray obtained in \cite{CFG-illinois} a $6$-dimensional complex nilmanifold with FSS not degenerating at the third page
(for the existence of balanced metrics on it, see Section~\ref{balanced-section} below).
Furthermore, for every $k\geq 2$, Bigalke and Rollenske constructed in \cite{BR} a $(4k-2)$-dimensional complex nilmanifold with $E_k\neq E_{\infty}$, which turns out to be balanced by a recent result by Sferruzza and Tardini \cite{ST21}.

Each of the nilmanifolds above occurs as a concrete example in a certain complex dimension.
In fact, as far as we know, there are no infinite families of manifolds in the literature, in the sense of having infinitely many different complex homotopy types, all living in the same complex dimension with $E_k\neq E_{\infty}$ for some $k$. %This is our main goal in this paper. 
The main goal of this paper is to construct such families. %prove that such families exist. %even such kind of behaviour can occur.
%provide such kind of infinite families.
The starting point of our construction will be the so-called \emph{strongly non-nilpotent} complex structures on nilmanifolds, recently studied in \cite{LUV1} and classified in \cite{LUV3} in four complex dimensions.

%In this paper, we show that one can indeed find infinite families of nilmanifolds (in the sense that they have infinitely many different homotopy types) with $E_2\ncong E_{\infty}$. To our knowledge, these are the first infinite manifolds in the literature with such behaviour.% in certain complex dimensions. 
%However, as far as we know, no infinite family in a given complex dimension appears in the literature.
%
%All the nilmanifolds above occur as {\color{red}a concrete example} {\color{blue} concrete (discrete) examples} in their respective complex dimensions
%{\color{red}, however as far as we know, no infinite family in a given complex dimension appears in the literature. Here by `infinite' we mean families of manifolds with infinitely many different homotopy types. 
%%(see Section~\ref{balanced-section} for details). 
%This is one of our main goals in the paper, for which} 
%{\color{blue}
%In this paper, we show that one can indeed find infinite families of nilmanifolds (in the sense that they have infinitely many different homotopy types) with balanced metrics and $E_2\ncong E_{\infty}$. To do this,
%}we consider the so-called \emph{strongly non-nilpotent} complex structures on nilmanifolds, recently studied in \cite{LUV1} and classified in \cite{LUV3} in four complex dimensions. 

Apart from the important role played by nilmanifolds in non-K\"ahler Hermitian geometry, there are some other reasons motivating the study of complex nilmanifolds in relation to the FSS. 
For instance, Kasuya proves in~\cite{Kasuya} that, in the larger class of solvmanifolds, if one considers those constructed from a semi-direct product of $\C^n$ by a nilpotent Lie group, then the page at which their FSS degenerates cannot be greater than that of the nilmanifold.
So, in this sense, nilmanifolds constitute a preferred class to search for compact complex manifolds with non-degenerate FSS. In addition, complex nilmanifolds, and in particular strongly non-nilpotent complex structures in four dimensions, have a remarkable role in relation to the problem of finding manifolds realizing certain generators of the universal ring
of cohomological invariants recently studied by Stelzig in~\cite{Stelzig-2021}.

We recall that the FSS of any complex nilmanifold $X$ with $\dim_{\C} X = 3$ is studied in \cite{COUV}; in particular, it is proved that the existence of a balanced metric on $X$ implies that $E_2(X) = E_{\infty}(X)$. Note that this is indeed a restriction, as there exist complex 3-dimensional nilmanifolds $X$ with 
$E_2(X) \not= E_3(X)$ (see \cite[Theorem 4.1]{COUV}). Therefore, complex dimension four is the lowest possible dimension for a balanced nilmanifold to have FSS not degenerating at the second page. 
In Theorem~\ref{cor-infinite-bal-FSS} we prove that there are infinitely many  
nilmanifolds satisfying these properties and with different complex (hence, real or rational) homotopy types. Moreover, this result is extended in Theorem~\ref{th-inf-balanced} to non-degeneration at any arbitrary page.

%\smallskip

In this paper we also deal with compact generalized Gauduchon manifolds, which were introduced and studied by Fu, Wang and Wu in \cite{FWW}. 
We recall that a Hermitian metric $F$ on a compact complex manifold $X$ with $\dim_{\C}X=n$ is called {\em $k$-th Gauduchon}, for some $1 \leq k \leq n -1$, if it satisfies the condition
$\partial \db F^k  \wedge F^{n -k-1} =0$. 
Observe that the value $k = n-1$ corresponds to the \emph{standard} (also known as \emph{Gauduchon}) metrics \cite{Gau}. Note also that any pluriclosed (or SKT) metric is in particular $1$-st Gauduchon, as $\partial \db F=0$.

We prove in Theorem~\ref{th-inf-k-G} that there are infinitely many generalized Gauduchon nilmanifolds with different complex homotopy type whose Fr\"olicher spectral sequence
can be arbitrarily non-degenerate.
Regarding pluriclosed metrics, Popovici proved in \cite{Pop1,Pop2} that the existence of a Hermitian metric on $X$ 
with `small torsion' implies $E_2(X)=E_{\infty}(X)$, 
and furthermore, he conjectured that any compact complex manifold 
$X$ admitting an SKT metric has
FSS degenerating at the second page \cite[Conjecture 1.3]{Pop1}.
 
In Proposition~\ref{counterex} we give a counterexample to this conjecture, based on the complex geometry of compact Lie groups. More concretely, 
we consider the compact semisimple Lie group SO(9) equipped with a left-invariant complex structure $J$ found by Pittie in \cite{Pittie-indian,Pittie-bull}, which is compatible with a bi-invariant metric $g$. 
Recall that any such compact Lie group is Bismut flat and its fundamental form $F$ is $dd^c$-harmonic by a result of Alexandrov and Ivanov \cite{AI}.

%\smallskip
The paper is structured as follows. 
In Section~\ref{FSS-section} we study the Fr\"olicher spectral sequence of $8$-dimensional nilmanifolds endowed with strongly non-nilpotent complex structures. 
A general study of the existence of balanced metrics on such complex nilmanifolds is given in Section~\ref{balanced-section}, from which we arrive at the results in Theorems~\ref{cor-infinite-bal-FSS} and~\ref{th-inf-balanced} mentioned above. 
Finally, Section~\ref{counterexample-section} is devoted to the FSS of compact pluriclosed and generalized Gauduchon manifolds.

%%%%%%%%%%%%%%%%%%%%%%%%%%%%%%%%%%%
%%%\section{Nilmanifolds with strongly non-nilpotent complex structure and non-degenerate Fr\"olicher sequence}\label{FSS-section}
\section{Complex nilmanifolds with $1$-dimensional center and non-degenerate Fr\"olicher spectral sequence}\label{FSS-section}
%%%\section{Fr\"olicher spectral sequence of complex nilmanifolds with $1$-dimensional center}\label{FSS-section}
%%%%%%%%%%%%%%%%%%%%%%%%%%%%%%%%%%%

\noindent In this section we study the FSS on $8$-dimensional nilmanifolds with one-dimensional center endowed with invariant complex structures. Infinite families of complex nilmanifolds with $E_2\neq E_{\infty}$ are obtained in complex dimension 4.

Let $X$ be a compact complex manifold with $\dim_{\C}X=n$. We recall that the Fr\"olicher spectral sequence of $X$ is 
the spectral sequence associated to 
%%%of the $\C$-valued de Rham complex $(\Omega^{*,*}_{\C}(X), d)$ given by 
the double complex $(\Omega^{*,*}(X), \partial,\bar\partial)$, 
where $\partial +\bar\partial = d$ is the usual decomposition of
the exterior differential $d$ on $X$.
This spectral sequence was first introduced in \cite{Fro}, in terms of a certain filtration, and it can be described as a collection of complexes 
$$
\dots\stackrel{d_r}{\longrightarrow} E_r^{p-r,\,q+r-1}(X)\stackrel{d_r}{\longrightarrow}E_r^{p,\,q}(X)\stackrel{d_r}{\longrightarrow}E_r^{p+r,\,q-r+1}(X)\stackrel{d_r}{\longrightarrow}\dots
$$ 
that are canonically associated with the complex structure of $X$, for every $r\geq 1$. 
The {\it $1$-st page} consists of the Dolbeault cohomology groups of $X$, i.e. $E_1^{p,\,q}(X)=H^{p,\,q}_{\bar\partial}(X)$, while the differentials $d_1$ are induced by $\partial$ as $d_1([\alpha]) = [\partial\alpha]$, for every Dolbeault class $[\alpha]\in H^{p,\,q}_{\bar\partial}(X)$. For an arbitrary~$r$, the differentials $d_r$ on the {\it $r$-th page} are of type $(r,\,-r+1)$ but they are still induced by $\partial$ acting on a certain $(p+r-1,q-r+1)$-form associated to any representative of every class in $E_r^{p,\,q}(X)$. (See the description below.) It turns out that $d_r\circ d_r=0$, and the {\it $(r+1)$-th page} is induced from the previous {\it $r$-th page} as the kernel of $d_r$ over the image of the incoming differential $d_r$.

There exists a positive integer $r$ from which all the differentials vanish identically, namely, $d_s=0$ for all $s\geq r$.
This is equivalent to having $E_r^{p,\,q}(X) = E_{r+k}^{p,\,q}(X)$ for every $k\geq 1$ and any $0\leq p,q\leq n$. This space $E_r^{p,\,q}(X)$ is denoted by $E_\infty^{p,\,q}(X)$ and the FSS is said to be {\it degenerated at the $r$-th page}, then writing $E_r(X) = E_\infty(X)$. 

The Fr\"olicher spectral sequence gives a link between the complex structure of $X$ and its differential structure. Indeed, it converges to the de Rham cohomology of $X$ in the sense that there are isomorphisms 
$H^k_{dR}(X,\,\C)\simeq\bigoplus\limits_{p+q=k}E_\infty^{p,\,q}(X)$, for every $k\in\{0,\dots , 2n\}$. 
%%%$$H^k_{DR}(X,\,\C)\simeq\bigoplus\limits_{p+q=k}E_\infty^{p,\,q}(X), \hspace{3ex} \mbox{ for every } k\in\{0,\dots , 2n\}.$$
%%%For more general results on the FSS, we refer the reader to \cite{Dem97},...
%%%Concerning general results on the FSS, 
It is worthy to recall that a Hodge theory is introduced in \cite{PSU} 
through the construction of elliptic pseudo-differential operators, 
associated with any given
Hermitian metric on $X$, whose kernels are isomorphic to the spaces $E_r^{p,q}(X)$
in every bidegree $(p, q)$. This extended to any arbitrary positive integer $r$ a previous construction in \cite{Pop1} for $r = 2$. We also remind that Serre duality for $E_r^{p,q}$ is proved by Stelzig in \cite{Stelzig}
%, by reducing it to the classical Serre duality for the first page, 
and by Milivojevi\'c in \cite{Mil19}. This duality is also obtained as a consequence of Hodge theory (see \cite{PSU} for more details).

\vskip.1cm

The following general description of the terms in the Fr\"olicher spectral sequence was given in~\cite{CFGU97} and it will be useful for our purposes. For every $r\geq 1$ and any $0\leq p,q\leq n$, the space $E^{p,q}_r(X)$ is isomorphic to the quotient $\C$-vector space
\begin{equation}\label{E=X/Y}
E^{p,q}_r(X)=\frac{{\mathcal X}^{p,q}_r(X)}{{\mathcal Y}^{p,q}_r(X)},
\end{equation}
where
$${\mathcal X}^{p,q}_1(X)= \{\alpha_{p,q} \in \Omega^{p,q}(X) \, \mid\,  \db\alpha_{p,q}=0 \},
\quad {\mathcal Y}^{p,q}_1(X)=\db\big(\Omega^{p,q-1}(X)\big),$$ 
and for every $r\geq 2$
\begin{equation}\label{Xpq}
\begin{array}{rl}
{\mathcal X}^{p,q}_r(X)= \{\alpha_{p,q} \in \Omega^{p,q}(X)  \mid\,  \!\!&\!\! \db\alpha_{p,q}=0, \mbox{ and there exist $r-1$ forms } \alpha_{p+1,q-1},\,\ldots\\[4pt]
& \ldots,\,\alpha_{p+r-2,q-r+2},\, \alpha_{p+r-1,q-r+1} \mbox{ satisfying }\\[5pt] 
0=\partial\alpha_{p,q}\!\!&\!\!+\,\db\alpha_{p+1,q-1}
=\cdots
=\partial\alpha_{p+r-2,q-r+2}+\db\alpha_{p+r-1,q-r+1} \},
\end{array}
\end{equation} 
%{\color{blue}\begin{equation*}
%\begin{split}
%&{\mathcal X}^{p,q}_r(X) = \{\alpha_{p,q} \in \Omega^{p,q}(X)  \mid \
%	  \db\alpha_{p,q}=0 \mbox{ and } \exists\,\alpha_{p+r-i,p-r+i}, \mbox{ for } i=1,\ldots,r-1,\\
%	&\text{ \ such that }  0=\partial\alpha_{p,q}+\db\alpha_{p+1,q-1}
%	=\cdots
%	=\partial\alpha_{p+r-2,q-r+2}+\db\alpha_{p+r-1,q-r+1} \},
%\end{split}
%\end{equation*} }
and
\begin{equation}\label{Ypq}
\begin{array}{rl}
{\mathcal Y}^{p,q}_r(X) = \{\db\beta_{p,q-1}+\partial\beta_{p-1,q} \in \Omega^{p,q}(X)  \mid  \!\!&\!\! \mbox{there exist $r-2$ forms } \beta_{p-2,q+1},\,\ldots\\[4pt]
\!&\! \ldots,\,\beta_{p-r+2,q+r-3},\, \beta_{p-r+1,q+r-2} \mbox{ satisfying }\\[5pt] 
0\!=\db\beta_{p-1,q}\!+\!\partial\beta_{p-2,q+1}
\!=\!\cdots\!=
\db\!\!&\!\!
\beta_{p-r+2,q+r-3}\!+\!\partial\beta_{p-r+1,q+r-2}
\!= \db\beta_{p-r+1,q+r-2} \}.
\end{array}
\end{equation} 
%{\color{blue}\begin{equation*}
%\begin{split}
%&{\mathcal Y}^{p,q}_r(X) = \{ \db\beta_{p,q-1}+\partial\beta_{p-1,q} \in \Omega^{p,q}(X) \mid 
%	\exists \beta_{p-r+i,q+r-i-1}, \mbox{ for } i=1,\ldots,r-2,\\
%	&\text{ \ with } 0=\db\beta_{p-1,q}\!+\!\partial\beta_{p-2,q+1}
%	\!=\!\cdots\!=
%	\db\beta_{p-r+2,q+r-3}\!+\!\partial\beta_{p-r+1,q+r-2} = \db\beta_{p-r+1,q+r-2} \}.
%\end{split}
%\end{equation*} }

Furthermore, the differentials $d_r\colon E_r^{p,\,q}(X)\longrightarrow E_r^{p+r,\,q-r+1}(X)$ are explicitly given by 
\begin{equation}\label{dr}
d_r\big( [\alpha_{p,q}]\big)= [\partial\alpha_{p+r-1,q-r+1}],
\end{equation}
for any $[\alpha_{p,q}]\in E_r^{p,\,q}(X)$. 

\vskip.2cm

Let $G$ be a simply
connected real Lie group endowed with a left-invariant complex structure $J$, 
and suppose that $G$ admits a discrete subgroup $\Gamma$ so that the quotient space $\nilm$ is compact. Let us denote by $X$ the latter manifold endowed with the (naturally induced) complex structure~$J$. 

Consider $\frg$, the Lie algebra of $G$, endowed with the (linear integrable) complex structure $J$. Then, 
we can define the corresponding sequence $E_r(\frg,J)$ associated to the pair $(\frg,J)$. The description~\eqref{E=X/Y} together with \eqref{Xpq} and \eqref{Ypq}, and so the homomorphisms \eqref{dr}, apply to this sequence. 

\begin{proposition}\label{NLA-a-nilvariedad}
Let $X=(\nilm,J)$ be a compact quotient of a simply connected Lie
group $G$ by a lattice $\Gamma$, endowed with a complex structure naturally induced by a left-invariant complex structure $J$ on $G$. Let $\frg$ be the Lie algebra of $G$. Fix an integer $r\geq 1$, and suppose that the homomorphism 
$d_r\colon E_r^{p,q}(\frg,J) \longrightarrow E_r^{p+r,q-r+1}(\frg,J)$ is non-zero for some $p,q$. 
Then, the Fr\"olicher spectral sequence of $X$ does not degenerate at the $r$-th page.
\end{proposition}

\begin{proof}
Let $[\alpha_{p,q}]$ be an element in $E_r^{p,q}(\frg,J)$ such that $d_r \big([\alpha_{p,q}]\big) \not=0$ in $E_r^{p+r,q-r+1}(\frg,J)$. Since $\alpha_{p,q}\in
{\mathcal X}^{p,q}_r(\frg,J)$ and there is a natural inclusion 
$\iota\colon{\mathcal X}^{p,q}_r(\frg,J)\hookrightarrow {\mathcal X}^{p,q}_r(X)$, the form $\alpha_{p,q}$ defines an element $[\alpha_{p,q}]$ in $E_r^{p,q}(X)$. 
Notice that we can choose the $r-1$ forms $\alpha_{p+1,q-1},\,\ldots,\,\alpha_{p+r-2,q-r+2},\, \alpha_{p+r-1,q-r+1}$ in \eqref{Xpq} to be left-invariant. 

Suppose that the FSS of $X$ degenerates at $r$-th page. Then, $d_r\big([\alpha_{p,q}]\big) =0$ in $E_r^{p+r,q-r+1}(X)$. 
By \eqref{Xpq} and \eqref{dr}, together with \eqref{E=X/Y}, this means that $\partial\alpha_{p+r-1,q-r+1}\in {\mathcal Y}^{p+r,q-r+1}_r(X)$.
From the description \eqref{Ypq}, there exist 
$r$ forms 
$$
\beta_{p+r,q-r},\ \beta_{p+r-1,q-r+1},\ \beta_{p+r-2,q-r+2},\ldots,\beta_{p+2,q-2},\ \beta_{p+1,q-1}
$$ 
on the complex manifold $X$ satisfying 
$$
\partial\alpha_{p+r-1,q-r+1}=\db\beta_{p+r,q-r}+\partial\beta_{p+r-1,q-r+1},$$  
and
$$
\db\beta_{p+r-1,q-r+1}+\partial\beta_{p+r-2,q-r+2}=0, \
\ldots, \ \ 
\db\beta_{p+2,q-2}+\partial\beta_{p+1,q-1}=0, \ \ 
\db\beta_{p+1,q-1}=0. $$
As the Lie group $G$ has a lattice, $G$ is unimodular. In particular, $G$
admits a bi-invariant volume form, so we can apply the well-known symmetrization process (see for instance \cite{COUV}  and the references therein for details). 
Given any form $\beta$ on $X$, we denote by $\tilde{\beta}$ the left-invariant form on $G$ given by the symmetrization of $\beta$. Recall that $J$ being left-invariant, the bidegree of the forms is preserved, and one has $\widetilde{\partial\beta}=\partial\tilde{\beta}$ and $\widetilde{\db\beta}=\db\tilde{\beta}$. 

Note that $\partial\alpha_{p+r-1,q-r+1}$ coincides with its symmetrization because it is left-invariant. 
Therefore, from the equalities above, we get $r$ left-invariant forms 
$$
\tilde{\beta}_{p+r,q-r},\ \tilde{\beta}_{p+r-1,q-r+1},\ \tilde{\beta}_{p+r-2,q-r+2},\ldots,\tilde{\beta}_{p+2,q-2},\ \tilde{\beta}_{p+1,q-1}
$$ 
satisfying 
$$
\partial\alpha_{p+r-1,q-r+1}=\db\tilde{\beta}_{p+r,q-r}+\partial\tilde{\beta}_{p+r-1,q-r+1},$$  
and
$$
\db\tilde{\beta}_{p+r-1,q-r+1}+\partial\tilde{\beta}_{p+r-2,q-r+2}=0, \
\ldots, \ \
\db\tilde{\beta}_{p+2,q-2}+\partial\tilde{\beta}_{p+1,q-1}=0, \ \ 
\db\tilde{\beta}_{p+1,q-1}=0. $$
But this implies  $\partial\alpha_{p+r-1,q-r+1}\in {\mathcal Y}^{p+r,q-r+1}_r(\frg,J)$, which is a contradiction to the hypothesis that $d_r \big([\alpha_{p,q}]\big)$ is non-zero in $E_r^{p+r,q-r+1}(\frg,J)$.

In conclusion, $d_r\big([\alpha_{p,q}]\big) \not=0$ in $E_r^{p+r,q-r+1}(X)$, and the FSS of $X$ does not degenerate at the $r$-th page.
\end{proof}

From now on in this section 
we will consider that $G$ is a nilpotent Lie group, and thus
 $X=(\nilm,J)$ is a complex nilmanifold. It should be noticed that 
if the natural map 
$\iota\colon\Lambda^{*,*}(\frg,J)\hookrightarrow \Omega^{*,*}(X)$ induces an isomorphism in Dolbeault cohomology, then 
one has also an isomorphism $E_r^{p,q}(X)\cong E_r^{p,q}(\frg,J)$ for every $k\geq 1$ and any $0\leq p,q\leq n$ (see also \cite[Lemma 7.5]{Stelzig}). 
This is indeed the case in complex dimension three \cite{R2, FRR}, or in arbritary dimension for any \emph{nilpotent} complex structure \cite{RTW}. %%%in the sense of \cite{Cordero-F-G-U-TransAMS}.

However, the complex structures $J$ on the nilmanifolds $X$ that we will study in this section are (strongly) non-nilpotent, so Proposition~\ref{NLA-a-nilvariedad} will be applied to derive the non-degeneration of the FSS of $X$ by means of the non-degeneration of the sequence $E_r(\frg,J)$. Observe that the proof of  Proposition~\ref{NLA-a-nilvariedad} implies that the natural inclusion 
$\iota$ induces an injection $E_r^{p,q}(\frg,J) \hookrightarrow  E_r^{p,q}(X)$ that commutes with the differentials $d_r$. 

\vskip.2cm

Let $\frg$ be a nilpotent Lie algebra (NLA for short) endowed with a complex structure~$J$. Since~$\frg$ is nilpotent, its center is always non-trivial. If we now assume that the center has dimension~1, then the only $J$-invariant subspace in it is the trivial one.
Complex structures having the latter property are known as \emph{strongly non-nilpotent} (SnN for short), and they are studied in \cite{LUV1}. 
We recall that, up to real dimension 8, the conditions \emph{``$J$ being SnN''} and \emph{``$\frg$ having $1$-dimensional center''} are equivalent. 
Furthermore, it is proved in \cite[Proposition 3.1]{LUV3} that 
when $\frg$ has dimension 8 and $J$ is of the previous type,
the dimension of the space $E^{0,1}_{1}(\frg,J)$ 
is either $2$ or $3$. 
This provides a partition of the space of SnN complex structures $J$ into two families.

\begin{definition}\label{familias-I-II}%\cite{LUV3}
{\rm
Let $\frg$ be an $8$-dimensional NLA endowed with an SnN complex structure $J$.
We say that $J$ belongs to \emph{Family I} (resp. \emph{Family II}) if $E^{0,1}_{1}(\frg,J)$ has maximal dimension (resp. minimal dimension).
}
\end{definition}

The following classification of SnN complex structures is available in \cite{LUV3}:

\begin{proposition}\label{main-theorem} 
%{\textbf{(Classification of {\color{blue} SnN} complex structures)}} 
\cite[Theorem 3.3]{LUV3}
\noindent Let $J$ be a complex structure on an $8$-dimensional NLA $\frg$ with one dimensional center.
Then, there exists a basis of $(1,0)$-forms $\{\omega^k\}_{k=1}^4$ in terms of which the complex structure equations of $(\frg,J)$ 
are one (and only one) of the following:
\begin{itemize}
\item[(i)] if $J$ belongs to {\textbf{Family I}}, then
\begin{equation}\label{ecus-I}
%\text{Family I:} \quad
\left\{
\begin{split}
d\omega^1 &= 0,\\[-4pt]
d\omega^2 &= \varepsilon\,\omega^{1\bar 1},\\[-4pt]
d\omega^3 &= \omega^{14}+\omega^{1\bar 4}+a\,\omega^{2\bar 1}+ i\,\delta\,\varepsilon\,b\,\omega^{1\bar 2},\\[-4pt]
d\omega^4 &= i\,\nu\,\omega^{1\bar 1} +b\,\omega^{2\bar 2}+ i\,\delta\,(\omega^{1\bar 3}-\omega^{3\bar 1}),
\end{split}
\right.
\end{equation}
where $\delta=\pm 1$, $(a,b)\in \mathbb R^2-\{(0,0)\}$ with $a\geq 0$,
and 
$(\varepsilon, \nu, a, b)$ being one of the following:
$(0,0,0,1)$, $(0,0, 1, 0)$, $(0,0, 1, 1)$, $(0,1, 0, \pm 1)$, $(0,1, 1,b)$, 
$(1,0, 0,1)$, $(1,0, 1,|b|)$ or $(1,1, a, b)$.

%\vskip.1cm

\item[(ii)] if $J$ belongs to {\textbf{Family II}}, then
\begin{equation}\label{ecus-II}
%%%\text{Family II:} \quad
\left\{
\begin{split}
d\omega^1&=0,\\[-4pt]
d\omega^2&=\omega^{14}+\omega^{1\bar 4},\\[-4pt]
d\omega^3&=a\,\omega^{1\bar 1}
                 +\varepsilon\,(\omega^{12}+\omega^{1\bar 2}-\omega^{2\bar 1})
                 +i\,\mu\,(\omega^{24}+\omega^{2\bar 4}),\\[-4pt]
d\omega^4&=i\,\nu\,\omega^{1\bar 1}-\mu\,\omega^{2\bar 2}+i\,b\,(\omega^{1\bar 2}-\omega^{2\bar 1})+i\,(\omega^{1\bar 3}-\omega^{3\bar 1}),
\end{split}
\right.
\end{equation}
where $a, b\in\mathbb R$, and the tuple $(\varepsilon, \mu, \nu, a, b)$ takes the following values: $(1, 1, 0, a, b)$, $(1, 0, 1, a, b)$, $(1, 0, 0, 0, b)$, 
$(1, 0, 0, 1, b)$, $(0, 1, 0, 0, 0)$ or $(0, 1, 0, 1, 0)$.
\end{itemize}
\end{proposition}

\vskip.2cm

Our first goal is to provide (the first known) examples of nilmanifolds endowed with SnN complex structures such that the differential $d_2\neq0$. We recall that up to real dimension 6 all such complex nilmanifolds have FSS degenerating at the second page 
(see \cite[Theorem 4.1]{COUV} for the NLAs $\frh_{19}^{-}$ and $\frh_{26}^{+}$).

\vskip.2cm

Next %we will focus on the case $p=0$ and $q=2$, and 
we will study the spaces $E^{0,2}_r(\frg,J)$ for every SnN $J$ on any 8-dimensional NLA $\frg$. 
One reason for focusing on the bidegree $(p,q)=(0,2)$ is motivated by the recent paper by Stelzig~\cite{Stelzig-2021}, where this bidegree plays an important role in complex dimension 4 in relation to the problem of finding manifolds realizing certain generators of the universal ring
of cohomological invariants in degree~4 (see \cite[Problem 11.2]{Stelzig-2021}). 

Moreover, note that one can take advantage of the explicit description \eqref{E=X/Y}-\eqref{dr} when particularized to any bidegree $(p,q)$ with $p=0$, since then ${\mathcal Y}^{0,q}_r=\{\db \beta_{0,q-1} \}$ for every $r\geq 1$. 
In particular, for $q=2$ the corresponding terms for the pair $(\frg,J)$ can be described as follows:
\begin{equation}\label{E-0-2}
E^{0,2}_r(\frg,J)=\frac{{\mathcal X}^{0,2}_r(\frg,\!J)}{\db\left(\Lambda^{0,1}(\frg,\!J)\right)},
\end{equation}
where 
$
{\mathcal X}^{0,2}_1(\frg,\!J)= \{\alpha_{0,2} \in \Lambda^{0,2}(\frg,\!J) \, \mid\, \db\alpha_{0,2}=0 \}
$, and
\begin{equation}\label{X-0-2}
\begin{array}{rl}
&{\mathcal X}^{0,2}_2(\frg,\!J)= \{\alpha_{0,2} \!\in\! {\mathcal X}^{0,2}_1(\frg,\!J)  \mid  \partial\alpha_{0,2}\!+\!\db\alpha_{1,1}\!=\!0 \mbox{ for some } \alpha_{1,1} \}, \\[5pt]
&{\mathcal X}^{0,2}_3(\frg,\!J)= \{\alpha_{0,2} \!\in\! {\mathcal X}^{0,2}_1(\frg,\!J)  \mid  \partial\alpha_{0,2}\!+\!\db\alpha_{1,1}
\!=\!\partial\alpha_{1,1}\!+\!\db\alpha_{2,0}\!=\!0 \mbox{ for some } \alpha_{1,1} \mbox{ and } \alpha_{2,0} \}, 
\end{array}
\end{equation}
whereas ${\mathcal X}^{0,2}_{r}(\frg,\!J)={\mathcal X}^{0,2}_{4}(\frg,\!J)$ for any $r\geq4$, with
$$
{\mathcal X}^{0,2}_{4}(\frg,\!J)= \{\alpha_{0,2} \in {\mathcal X}^{0,2}_1(\frg,\!J)  \mid  \partial\alpha_{0,2}\!+\!\db\alpha_{1,1}\!=\!\partial\alpha_{1,1}\!+\!\db\alpha_{2,0}\!=\!\partial\alpha_{2,0}\!=\!0 \mbox{ for some } \alpha_{1,1} \mbox{ and } \alpha_{2,0} \}.
$$
%As we recall above, the property $E_2(\frg,J) \not= E_{\infty}(\frg,J)$ is equivalent to $d_2\not=0$. 
Moreover, $d_r\equiv 0$ for $r\geq 4$, and for $1\leq r\leq 3$ we have that $d_r\colon E^{0,2}_r(\frg,J) \longrightarrow E^{r,3-r}_r(\frg,J)$ is defined by 
$$
d_r\big([\alpha_{0,2}]\big)=[\partial\alpha_{r-1,3-r}],
$$
for any $[\alpha_{0,2}]\in E^{0,2}_r(\frg,J)$. Note that 
$E^{0,2}_{r+1}(\frg,J)=\ker d_r$ because the incoming $d_r$ is identically zero by bidegree reasons.

\vskip.2cm

%%%%%%%%%%%%%%%%%%
%\subsection{Family I}\label{FSS-section-fam1}
%%%%%%%%%%%%%%%%%%

In the following result we study in detail the terms $E_r^{0,2}$ in the FSS of an interesting subclass of complex structures in the Family I. 

\begin{proposition}\label{no-deg-fam1-NLA}
Let $J$ be an SnN complex structure in Family I defined by $\varepsilon=1$ and $ab\neq 0$. 
Let $\Theta(\delta,\nu,a,b)=\big((a-b)^2-2\delta\nu b\big)\big((a+b)^2-2\delta\nu b\big)$. Then,  %$d^{0,2}_2\colon E^{0,2}_2 \longrightarrow E^{2,1}_2$ is not zero, since $[\omega^{\bar3\bar4}]\in E^{0,2}_2$ satisfies that  
%$$
%d^{0,2}_2([\omega^{\bar3\bar4}]_{E_2})=[...]_{E_2}\neq 0 {\mbox\ in}\  E^{2,1}_2.
%$$
%This implies that 
$$
E_2^{0,2}(\frg,J) \neq E_3^{0,2}(\frg,J) \ \Longleftrightarrow \ \Theta(\delta,\nu,a,b) \neq 0.
$$ 
Moreover, in this case we have:
\begin{equation*}
\begin{split}
E^{0,2}_1(\frg,J) &= E^{0,2}_2(\frg,J) = \langle[\omega^{\bar1\bar2}],\, [\omega^{\bar1\bar3}],\,  [\omega^{\bar2\bar4}], \, [\omega^{\bar3\bar4}]\rangle, \\[4pt]  %E^{0,2}_2 = \langle[\omega^{\bar1\bar2}],\, [\omega^{\bar1\bar3}],\, [\omega^{\bar2\bar4}],\, [\omega^{\bar3\bar4}]\rangle$$ 
%and
%$$
E^{0,2}_r(\frg,J) &= \langle[\omega^{\bar1\bar2}],\, [\omega^{\bar1\bar3}],\, [\omega^{\bar2\bar4}]\rangle,\ \ {\mbox  for} \ r\geq 3.
\end{split}
\end{equation*}
\end{proposition}

\begin{proof}
We use the description given in \eqref{E-0-2} and \eqref{X-0-2}
to compute the terms $E_r^{0,2}(\frg,J)$. From the complex structure equations \eqref{ecus-I} it follows that the space of $\bar\partial$-closed $(0,2)$-forms is given by $\mathcal X^{0,2}_1(\frg,J)= \langle \omega^{\bar1\bar2},\, \omega^{\bar1\bar3},\, \omega^{\bar1\bar4},\, \omega^{\bar2\bar4}, \, \omega^{\bar3\bar4} \rangle$. 
Since $\db\left(\Lambda^{0,1}(\frg,\!J)\right) = \langle \omega^{\bar1\bar4}\rangle$, we obtain the space $E_1^{0,2}(\frg,J)$ given in the statement above.  

Note that $\omega^{\bar1\bar2}$ is a $d$-closed form, so $[\omega^{\bar1\bar2}]\in E^{0,2}_r(\frg,J)$ for every $r\geq 2$.
Moreover, the classes $[\omega^{\bar1\bar3}]$ and $[\omega^{\bar2\bar4}]$ also belong to $E^{0,2}_r(\frg,J)$ for every $r\geq 2$, due to the following relations:
$$
%\begin{equation*}
\begin{array}{l l l l l l l l l l l l}
%%%\bar\partial \omega^{\bar1\bar3}&\!\!\!=\!\!\!& 0,&&&  \quad &\bar\partial \omega^{\bar2\bar4}&\!\!\!=\!\!\!& 0,&&&\\[5pt]
\partial \omega^{\bar1\bar3} &\!\!\!+\!\!\!&   \bar\partial (a\,\omega^{2\bar2})&\!\!\!=\!\!\!&0,&    \quad   &\partial \omega^{\bar2\bar4} &\!\!\!+\!\!\!&   \bar\partial (2i\,\nu\,\omega^{2\bar2}+ \omega^{2\bar4} - \omega^{4\bar2})&\!\!\!=\!\!\!&0,&\\[5pt]
&& \partial (a\,\omega^{2\bar2})&\!\!\!+\!\!\!&\bar\partial (-\omega^{13})=  0,&   &&& \partial (2i\,\nu\,\omega^{2\bar2}+\omega^{2\bar4} - \omega^{4\bar2})&\!\!\!+\!\!\!&\bar\partial (\omega^{24}) =\!\!\!&0,\\[5pt]
&&&&\partial (-\omega^{13})=  0;&  &&&&&\partial (\omega^{24}) =\!\!\!&0.
\end{array}
%\end{equation*}
$$

%%%% OTRA POSIBLE CADENA PARA \omega^{\bar2\bar4} %%%%%%%%%%%%
%\begin{equation*}
%\begin{array}{llllll}
%\bar\partial \omega^{\bar2\bar4}& =& 0;&&&\\[5pt]
%\partial \omega^{\bar2\bar4} & + &   
%\bar\partial \big(i\,(2\,\nu-\delta\,b)\,\omega^{2\bar2} - \omega^{3\bar1} -
%	\omega^{4\bar2}\big)&=  &0;&\\[5pt]
%&& \partial \big(i\,(2\,\nu-\delta\,b)\,\omega^{2\bar2} - \omega^{3\bar1} -
%	\omega^{4\bar2}\big) &+&
%	\bar\partial \big(\frac{i\,(\delta\,b-2\,\nu)}{a}\,\omega^{13}\big)=  0;\\[5pt]
%&&&&\partial \big(\frac{i\,(\delta\,b-2\,\nu)}{a}\,\omega^{13}\big)=  0.
%\end{array}
%\end{equation*}

Let us focus now on the $(0,2)$-form $\omega^{\bar3\bar4}$. Since $ab\neq 0$, one can verify that $\partial \omega^{\bar3\bar4} + \bar\partial \gamma=0$ for the $(1,1)$-form $\gamma$ given by 
$$
\gamma = \frac{i\delta (-a^2 +b^2+ 2\delta \nu b)}{2b}\omega^{2\bar3} - \frac{a^2-b^2+2\delta \nu b}{2ab}\big(b\,\omega^{3\bar2} +a\,\omega^{4\bar1}\big) -\frac{i\delta(a^2+b^2-2\delta \nu b)}{2ab}\omega^{3\bar4} - \omega^{4\bar3}.
$$
Therefore, $[\omega^{\bar3\bar4}]\in E^{0,2}_2(\frg,J)$, and 
$E^{0,2}_1(\frg,J) = E^{0,2}_2(\frg,J)$. 

\medskip

Notice that any $(1,1)$-form $\alpha_{1,1}$ satisfying the condition $\partial \omega^{\bar3\bar4} + \bar\partial \alpha_{1,1}=0$ can be written as $\alpha_{1,1}=\gamma+\sigma$, with $\sigma \in {\mathcal X}^{1,1}_1(\frg,\!J)$, i.e. $\sigma$ is any $\bar\partial$-closed (1,1)-form. Since we have to study the solutions of the equation $\partial\alpha_{1,1} + \bar\partial \alpha_{2,0}=0$, for some form $\alpha_{2,0}$ of bidegree $(2,0)$, it is enough to consider the (1,1)-forms $\sigma$ in the quotient space
$$
{\mathcal V}(\frg,\!J)={\mathcal X}^{1,1}_1(\frg,\!J)/\{d\mbox{-closed } (1,1)\mbox{-forms} \}.
$$
Using \eqref{ecus-I} it is not difficult to see that 
${\mathcal V}(\frg,\!J) = \langle \omega^{1\bar3}+\omega^{2\bar4}, \omega^{1\bar4} \rangle$, 
so we can express $\sigma$ as  
$$\sigma =  c_1\,(\omega^{1\bar3}+\omega^{2\bar4})+c_2\,\omega^{1\bar4}
%+c_3\,(i\,\delta\,b\,\omega^{2\bar2}+ \omega^{2\bar4}+\omega^{3\bar1})
,\quad \text{where \,}c_1,c_2\in\C.$$ 
%------------------------
%
%{\color{red} La expresi\'on de $\tilde\gamma$ se puede simplificar tomando $c_3=0$:
%
%En efecto, definimos $$\tilde{\tilde\gamma} = \tilde\gamma + c_3 \alpha$$ siendo $\alpha = \omega^{1\bar3} - \omega^{3\bar1} - i\delta b\omega^{2\bar2}$.  Como $d\alpha = 0$, se tiene que $d\tilde\gamma = d\tilde{\tilde\gamma}$.  Observar:
%$$\tilde{\tilde\gamma} = \tilde\gamma + c_3 \alpha = (c_1+c_3) (\omega^{1\bar3} + \omega^{2\bar4}) + c_2\,\omega^{1\bar4},$$ es decir, encontramos una expresi\'on como $ \tilde\gamma $ con $c_3=0$.
%
%\bigskip
%
%Adela, si est\'as de acuerdo con esto, las sumas $c_1+c_3$ se podr\'ian sustituir simplemente por $c_1$ en lo que viene debajo.
%}
%------------------------
Then, we get 
\begin{eqnarray*}
\partial (\gamma + \sigma) %&=& 
&\!\!=\!\!& i\,(\nu-\delta b)\,c_1\,\omega^{12\bar1} 
+ \frac{i\,\delta\,a (a^2 - b^2-2\delta\nu  b) +2\,b^2\,c_2}{2b}\,
	\omega^{12\bar2} 
-i\delta\, c_1\, \omega^{12\bar3} \\
&&+\frac{a^2\,(b+\delta\nu)-b(b-2\delta\nu)(b-\delta\nu)-2\,i\,\delta\,a\,b\,c_2}{2ab}\, 
	\omega^{13\bar1}  
-\frac{a^2+b^2-2\delta \nu b}{2ab}\,\omega^{13\bar3}\\
&& + c_1\,\omega^{14\bar1} 
+ \frac{a^2+b^2-2\delta \nu b}{2a}\,\omega^{14\bar2}  
- \frac{i\delta(a^2+b^2-2\delta \nu b)}{2ab}\,\omega^{14\bar4} \\
&&   -i\delta \,c_1\,\omega^{23\bar1}
+ \frac{i\delta(a^2+b^2-2\delta \nu b)}{2a}\,\omega^{23\bar2}
- \frac{i\delta (a^2+b^2 - 2\delta \nu b)}{2b}\omega^{24\bar1}.
\end{eqnarray*} 

\medskip

Now, observe that any $(2,0)$-form $\alpha_{2,0}$ can be written as $\alpha_{2,0} = \sum_{1\leq r< s\leq 4}\lambda_{rs}\,\omega^{rs}$, with $\lambda_{rs}\in\C$. A direct calculation shows 
%$$C_0\,\omega^{12} + C_1\,\omega^{13} + C_2\,\omega^{14} + C_3\,\omega^{23} + C_4\,\omega^{24} + C_5\,\omega^{34},$$
\begin{eqnarray*}
\bar\partial\alpha_{2,0} &\!\!=\!\!& (-a \lambda_{13} + i\,\nu\,\lambda_{24})\omega^{12\bar 1} -b\,(\lambda_{14} - i\,\delta\,\lambda_{23})\omega^{12\bar 2} + i\,\delta\,\lambda_{24}\omega^{12\bar 3} + \lambda_{23}\omega^{12\bar 4}\\ 
&& +\, i\,(\delta\,\lambda_{14}+i\,\lambda_{23} + \nu\,\lambda_{34})\omega^{13\bar 1}
+  i\,\delta\,\lambda_{34}\omega^{13\bar3} -\lambda_{24}\omega^{14\bar1}-  i\,\delta\,b\,\lambda_{34}\omega^{14\bar2}\\ 
&& -\lambda_{34}\omega^{14\bar4} + i\,\delta\,\lambda_{24}\omega^{23\bar1} + b\,\lambda_{34}\omega^{23\bar2} - a\,\lambda_{34}\omega^{24\bar1}.
\end{eqnarray*}

\medskip

We have to study the equation $0=\partial (\gamma + \sigma) - \bar\partial\alpha_{2,0}=\sum A_{rs\bar k}\,\omega^{rs\bar k}$, where the coefficients $A_{rs\bar k}$ can be directly obtained from the previous expressions.
In particular, we will obtain some relations among $c_1$, $c_2$, $\lambda_{rs}$ and the parameters defining the complex structure $J$ that ensure $A_{rs\bar k}=0$ for all $r,s,k$.
First, observe that from $A_{12\bar4}=A_{14\bar1}=A_{14\bar4}=0$ one gets
$$\lambda_{23}=0, \qquad
\lambda_{24}=-c_1, \qquad
\lambda_{34}=\frac{i\,\delta\,(a^2+b^2-2\,\delta\,\nu\,b)}{2ab}.
$$
These values make most of the coefficients $A_{rs\bar k}$ to vanish, with the exception of $A_{12\bar1}$, $A_{12\bar2}$ and $A_{13\bar1}$. 
From the vanishing of the first and second aforementioned coefficients, one obtains
$$\lambda_{13}=\frac{i\,c_1\,(\delta\,b-2\nu)}{a}, \qquad
\lambda_{14}=-\frac{i\,\delta\,a\,(a^2-b^2-2\,\delta\,\nu\,b)+2\,b^2\,c_2}{2\,b^2}.$$
This immediately gives that $A_{13\bar1}=0$ if and only if
$$
%\begin{equation}\label{obstruction}
\big((a-b)^2-2\delta \nu b\big)\big((a+b)^2-2\delta \nu b\big)=0.
%\end{equation} 
$$

Consequently, if $\Theta(\delta,\nu,a,b)\neq0$ 
then $[\omega^{\bar3\bar4}]\in E^{0,2}_2(\frg,J)$, but $[\omega^{\bar3\bar4}]\notin E^{0,2}_3(\frg,J)$. However, if $\Theta(\delta,\nu,a,b)=0$ then $[\omega^{\bar3\bar4}]\in E^{0,2}_r(\frg,J)$ for every $r\geq 3$, because $\lambda_{23}=0$ implies that $\partial\alpha_{2,0}=0$.
This completes the proof of the proposition.
\end{proof}

%\begin{remark}
%Condicion~\eqref{obstruction} is satisfied in cases $(\varepsilon, \nu, a,b) = (1,0,1,1)$ (Lie algebra $\frg_4^{1,1}$) and  $(\varepsilon, \nu, a,b) = (1,1,0,2\delta)$ (Lie algebra $\frg_5$).
%\end{remark}

%%%%%%%%%%%%%%%%%%
%\subsection{Family II}\label{FSS-section-fam2} 
%%%%%%%%%%%%%%%%%%

In the following result an interesting subclass of structures in the Family II is studied. 

\begin{proposition}\label{no-deg-fam2-NLA}
Let $J$ be an SnN complex structure in Family II defined by $\varepsilon=\mu=1$ (hence $\nu=0$). Then, $E_2^{0,2}(\frg,J) \not= E_3^{0,2}(\frg,J)$. 
Moreover,
\begin{equation*}
\begin{split}
E^{0,2}_1(\frg,J) &= E^{0,2}_2(\frg,J) = \langle [i\,\omega^{\bar{1}\bar{2}} \!-\omega^{\bar{2}\bar{4}}], [\omega^{\bar{1}\bar{3}}\!- i\,\omega^{\bar{3}\bar{4}}] \rangle, \\[4pt] 
E^{0,2}_r(\frg,J) &= \langle [i\,\omega^{\bar{1}\bar{2}}\!-\omega^{\bar{2}\bar{4}}] \rangle,\ \ {\mbox  for} \ r\geq 3.
\end{split}
\end{equation*}
\end{proposition}

\begin{proof}
We use again the description given in \eqref{E-0-2} and \eqref{X-0-2} to compute the desired terms of the FSS. %We first observe that $\varepsilon=\mu=1$ implies $\nu=0$.
From the complex structure equations~\eqref{ecus-II} it follows that the space of $\bar\partial$-closed $(0,2)$-forms is given by ${\mathcal X}^{0,2}_1(\frg,J)=\langle \omega^{\bar{1}\bar{2}}, \omega^{\bar{1}\bar{4}},\omega^{\bar{2}\bar{4}}, \omega^{\bar{1}\bar{3}}-i\,\omega^{\bar{3}\bar{4}} \rangle$.
Since $\db\left(\Lambda^{0,1}(\frg,\!J)\right) = \langle \omega^{\bar{1}\bar{4}}, 
\omega^{\bar{1}\bar{2}}-i\,\omega^{\bar{2}\bar{4}} \rangle$, %and $(\varepsilon,\mu)\not=(0,0)$, 
we obtain the space $E_1^{0,2}(\frg,J)$ in the statement.
%$$
%E_1^{0,2}(\frg,J)=\langle [i\,\omega^{\bar{1}\bar{2}} -\omega^{\bar{2}\bar{4}}], 
%	[\omega^{\bar{1}\bar{3}}-i\,\omega^{\bar{3}\bar{4}}] \rangle.
%$$

%From now on, we focus on complex structures $J$ defined by \eqref{ecus-II} with $\varepsilon=\mu=1$. \footnote{No s\'e si tiene mucho sentido empezar con $\varepsilon$ libre si en el enunciado ya estamos en el caso $\varepsilon=1$} Notice that this implies $\nu=0$. 
Now, since 
\begin{equation*}
\begin{array}{llllll}
%%%\bar\partial (i\,\omega^{\bar1\bar2} - \omega^{\bar2\bar4})& =& 0;&&\\[5pt]
\partial (i\,\omega^{\bar1\bar2} - \omega^{\bar2\bar4}) & \!\!+\!\! &   \bar\partial\, \omega^{4\bar2}&\!\!=\!\!&0,\\[4pt]
&& \partial\, \omega^{4\bar2}&\!\!=\!\!&0,
\end{array}
\end{equation*}
we conclude that the $(0,2)$-form $i\,\omega^{\bar{1}\bar{2}} -\omega^{\bar{2}\bar{4}}$ defines a non-zero class in $E^{0,2}_r(\frg,J)$ for every $r$. 

\medskip

Let us now consider the form $\omega^{\bar{1}\bar{3}}-i\,\omega^{\bar{3}\bar{4}}$.  One can check that $\partial(\omega^{\bar{1}\bar{3}}-i\,\omega^{\bar{3}\bar{4}}) + \db\gamma=0$, where 
$$
\gamma = \frac{3ia}{2}\,\omega^{1\bar{2}} +\frac{1}{2}\,\omega^{1\bar{3}} +\frac{1}{2}\,\omega^{3\bar{1}} +\frac{i}{2}\,\omega^{3\bar{4}} +i\,\omega^{4\bar{3}},
$$
therefore 
$
E^{0,2}_2(\frg,J)=E^{0,2}_1(\frg,J)$. To finish the proof we will show that the form  $\omega^{\bar{1}\bar{3}}-i\,\omega^{\bar{3}\bar{4}}$ does not  belong to the space ${\mathcal X}^{0,2}_3(\frg,J)$.
As in the proof of Proposition~\ref{no-deg-fam1-NLA}, this is equivalent to prove that $\partial(\gamma+\sigma)\notin \db\left(\Lambda^{2,0}(\frg,\!J)\right)$ for every $(1,1)$-form $\sigma$ representing a class in the quotient space
$$
{\mathcal V}(\frg,\!J)={\mathcal X}^{1,1}_1(\frg,\!J)/\{d\mbox{-closed } (1,1)\mbox{-forms} \}.
$$
In other words, next we will prove that $\partial(\gamma+\sigma) + \bar\partial \alpha_{2,0}\neq0$ for every such $\sigma$ and every form $\alpha_{2,0}$ of bidegree $(2,0)$.

Using \eqref{ecus-II} one can see that 
${\mathcal V}(\frg,\!J) = \langle \omega^{1\bar4},\omega^{2\bar4} \rangle$, 
so we can express $\sigma$ as  
$\sigma =  c_1\,\omega^{1\bar4}+c_2\,\omega^{2\bar4}
%+c_3\,(i\,\delta\,b\,\omega^{2\bar2}+ \omega^{2\bar4}+\omega^{3\bar1})
$, where $c_1,c_2\in\C$.  
A direct computation shows that
\begin{eqnarray*}
\partial (\gamma + \sigma) %&=& 
&\!\!=\!\!& (1\!-ibc_1)\omega^{12\bar{1}} \!- (c_1+ibc_2)\omega^{12\bar2} \!- ic_2\,\omega^{12\bar3}
+\frac i2\omega^{12\bar{4}} \!-ic_1\,\omega^{13\bar1}+ \frac b2\,\omega^{13\bar{2}} +\frac12 \omega^{13\bar{3}} \\
&&  +\frac{ia}{2}\,\omega^{14\bar{1}} +\frac i2\,\omega^{14\bar{2}} + c_2\,\omega^{14\bar4} - \frac{b+2ic_2}{2}\omega^{23\bar{1}}+ \frac i2\omega^{23\bar{2}} -\frac i2\,\omega^{24\bar{1}} -\frac12\omega^{24\bar{4}}.
\end{eqnarray*} 

Now, writing any (2,0)-form $\alpha_{2,0}$ as $\alpha_{2,0} = \sum_{1\leq r< s\leq 4}\lambda_{rs}\,\omega^{rs}$, with $\lambda_{rs}\in\C$, from \eqref{ecus-II} we get 
\begin{eqnarray*}
\bar\partial\alpha_{2,0} &\!\!=\!\!& (\lambda_{13} + ib\lambda_{14}+a\lambda_{23})\omega^{12\bar 1} +(\lambda_{14}+\lambda_{23}+ ib\lambda_{24})\omega^{12\bar 2} + i\,\lambda_{24}\omega^{12\bar 3} -i\, \lambda_{13}\omega^{12\bar 4}\\[4pt] 
&& +\, i\lambda_{14}\omega^{13\bar 1}
+  ib\lambda_{34}\omega^{13\bar2} +i\lambda_{34}\omega^{13\bar3}- \lambda_{23}\omega^{13\bar4} -a\lambda_{34}\omega^{14\bar1}-\lambda_{34}\omega^{14\bar2}\\[4pt] 
&&-\lambda_{34}\omega^{14\bar4} + i(\lambda_{24}-b\lambda_{34})\omega^{23\bar1} -\lambda_{34}\omega^{23\bar2} + \lambda_{34}\omega^{24\bar1}-i\lambda_{34}\omega^{24\bar4}.
\end{eqnarray*}

Let us study the condition $0=\partial (\gamma + \sigma) - \bar\partial\alpha_{2,0}=\sum A_{rs\bar k}\,\omega^{rs\bar k}$, where the coefficients $A_{rs\bar k}$ are obtained directly from the previous expressions. From $A_{13\bar1}=A_{13\bar4}=0$ one gets $\lambda_{14}=-c_1$ and $\lambda_{23}=0$. Hence, $A_{12\bar1}=0$ is equivalent to $\lambda_{13}=1$. 
However, $A_{12\bar4}=0$ gives $\lambda_{13}=-\frac 12$, which is a contradiction. 
\end{proof}

We recall that, by the classification obtained in \cite[Theorem 1.1]{LUV3}, an 8-dimensional NLA $\frg$ with one dimensional center admits a complex structure if and only if it is isomorphic to one (and only one) in the following list:

\vskip.3cm

$
\begin{array}{rl}
&\mathfrak \frg_{1}^{\gamma} = (0^5,\, 13+15+24,\, 14-23+25,\, 16+27+\gamma\!\cdot\! 34), \
	\text{where }\gamma\in\{0,1\},\\[6pt]
&\mathfrak \frg_2^{\alpha} = (0^4,\, 12,\, 13+15+24,\, 14-23+25,\, 16+27+\alpha\!\cdot\! 34), \
	\text{where } \alpha\in\mathbb R, \\[6pt]
&\mathfrak \frg_3^{\gamma} = (0^4,\,12,\,
	13+\gamma\!\cdot\! 15+25,\, 15+24+\gamma\!\cdot\! 25,\, 16+27), \
	\text{ where } \gamma\in\{0,1\}, \\[6pt]
&\mathfrak \frg_4^{\alpha,\,\beta} = (0^4,\,12,\, 15+(\alpha\!+\!1)\!\cdot\! 24,\,
	\, (\alpha\!-\!1)\!\cdot\! 14-23+(\beta\!-\!1)\!\cdot\! 25,\,
	16+27+34-2\!\cdot\! 45), \\[4pt]
	& \hskip1cm \text{ where } (\alpha, \beta)\in\mathbb R^*\times \R^+ \text{ or }\
	\R^+\times \{0\}, \\[6pt]
&\mathfrak \frg_5 = (0^4,\,2\!\cdot\! 12,\,14-23,\,13+24,\,16+27+35),\\[6pt]
&\mathfrak \frg_6 = (0^4,\,2\!\cdot\! 12,\,14+15-23,\,13+24+25,\,16+27+35),\\[6pt]
&\mathfrak \frg_7 = (0^5,\,15,\,25,\,16+27+34),\\[6pt]
&\mathfrak \frg_8 = (0^4,\,12,\,15,\,25,\,16+27+34),\\[6pt]
&\mathfrak \frg_9^{\gamma} = (0^3,\, 13,\, 23,\, 35,\, \gamma\!\cdot\! 12-34,\, 16+27+45),\
	\text{where }\gamma\in\{0,1\},\\[6pt]
&\mathfrak \frg_{10}^{\gamma} = (0^3,\, 13,\, 23,\, 14+25,\, 15+24,\, 16+ \gamma\!\cdot\! 25 +27),\
	\text{where }\gamma\in\{0,1\},\\[6pt]
&\mathfrak \frg_{11}^{\alpha, \beta} = \big(0^3,\, 13,\, 23,\, 14+25-35,\, \alpha \!\cdot\! 12+15+24+34,\, 16+27-45-\beta(2\!\cdot\! 25 + 35)\big),\\[6pt]
& \hskip1cm \text{ where } (\alpha, \beta)=(0,0), (1,0) \text{ or } (\alpha, 1) \text{ with } \alpha\in[0,+\infty), \\[6pt]
&\mathfrak \frg_{12}^{\gamma} = (0^2,\, 12,\, 13,\, 23,\, 14+25,\, 15+24,\, 16+27+ \gamma\!\cdot\! 25),\
	\text{where }\gamma\in\{0,1\}.
\end{array}
$

\vskip.4cm

In the description of the nilpotent Lie algebras above we are using the standard abbreviated notation,
where the $i$-th component of the tuple contains the differential of the $i$-th element
of the basis. We recall that complex structures on the NLAs $\frg_1^{\gamma},\ldots,\frg_8$ belong to Family I, whereas those on $\frg_9^{\gamma}$, $\frg_{10}^{\gamma}$, $\frg_{11}^{\alpha,\beta}$ and $\frg_{12}^{\gamma}$ belong to Family~II.
Moreover, the previous list is ordered according to the dimensions of the ascending central series of the algebras.

We recall that the precise relation between these 8-dimensional NLAs and the classification of complex structures showed in Proposition~\ref{main-theorem} can be found in~\cite{LUV3}.
Nevertheless, the essential information is gathered in the Tables~1 and 2 (columns 1, 2 and 5), 
where we also sum up the behaviour of the sequence $E_r^{0,2}(\frg,\!J)$ for any complex structure $J$ in the Families I and II, respectively (columns 3 and 4). Note that this completes the results obtained 
in Propositions~\ref{no-deg-fam1-NLA} and~\ref{no-deg-fam2-NLA}.
We omit the details here. 

\newpage
In both tables we denote by $e_r^{0,\,2}$ the dimension of the term $E_r^{0,\,2}(\frg,J)$ in the spectral sequence. In Table~1, the parameter $s$ stands for the sign of $b\!-\!2\delta \nu$. 

\vspace{1cm}

\renewcommand{\arraystretch}{1.2}
\setlength{\tabcolsep}{1pt}
\begin{center}
\resizebox{15cm}{!} {
\begin{tabular}{|c|c|c|c|c|}
%%%%%%%%%%%%%%%%%%%%%%%%%%%%%%%%%%%%%%%%%%%%%%%%%%%% CABECERA %
\hline
\begin{tabular}{c} Ascending \\[-4pt] type \end{tabular} & $J\!\equiv\!(\varepsilon,\nu,a,b)$ & 
\small{$\big(e_1^{0,2},e_2^{0,2},e_3^{0,2}\big)$} & Sequence $E_r^{0,2}(\frg,\!J)$ & NLA $\frg$ \\
\hline\hline
%%%%%%%%%%%%%%%%%%%%%%%%%%%%%%%%%%%%%%%%%%%%%%%%%%%%%% (1,3,8) %
$(1,3,8)$ 
& $\begin{array}{c}\big(0,0,1,b\big) \\[-2pt] b\in\{0,1\}\end{array}$
& (4,2,2)
& $E_1^{0,2}\neq E_2^{0,2}=E_{\infty}^{0,2}$
& $\mathfrak g_1^b$ \\
\hline\hline
%%%%%%%%%%%%%%%%%%%%%%%%%%%%%%%%%%%%%%%%%%%%%%%%%%%%% (1,3,6,8) %
\multirow{19}{*}{$(1,3,6,8)$} 
& $(0,1,1,\frac{\delta}{2})$
& $(4,3,3)$
& \multirow{3}{*}{$E_1^{0,2}\neq E_2^{0,2}=E_{\infty}^{0,2}$}
& \multirow{3}{*}{$\mathfrak g_2^{-4\delta b}$} \\[1pt]
\cline{2-3}
%%%%%%%%%%%%%%%%%%%%%%%%%%%%%%%%%%%%%%%%%%%%%%%%%%%%%%
% hueco	
& $\begin{array}{c}(0,1,1,b) \\[-2pt] b\in\mathbb R\!-\!\{\frac{\delta}{2}\}\end{array}$
& $(4,2,2)$
&
&  \\
\cline{2-5}
%%%%%%%%%%%%%%%%%%%%%%%%%%%%%%%%%%%%%%%%%%%%%%%%%%%%%%
% hueco	
& $\begin{array}{c}(1,1,a,0)\\[-2pt] a\in(0,2)\end{array}$
& \multirow{5}{*}{$(4,3,3)$}
& \multirow{5}{*}{$E_1^{0,2}\neq E_2^{0,2}=E_{\infty}^{0,2}$}
& $\mathfrak g_2^{0}$ \\
\cline{2-2}\cline{5-5}
%%%%%%%%%%%%%%%%%%%%%%%%%%%%%%%%%%%%%%%%%%%%%%%%%%%%%%
% hueco
& $(1,1,2,0)$
& 
&
& $\mathfrak g_3^{1}$ \\
\cline{2-2}\cline{5-5}
%%%%%%%%%%%%%%%%%%%%%%%%%%%%%%%%%%%%%%%%%%%%%%%%%%%%%%
% hueco
& $\begin{array}{c}(1,1,a,0)\\[-2pt] a\in(2,\infty) \end{array}$
&
&
& \multirow{2}{*}{$\mathfrak g_3^{0}$}\\%$\mathfrak n_3^{0}$\\
\cline{2-2}
%%%%%%%%%%%%%%%%%%%%%%%%%%%%%%%%%%%%%%%%%%%%%%%%%%%%%%
% hueco
& $(1,0,1,0)$
& 
&
&  \\%\multirow{3}{*}{$\mathfrak n_3^{0}$} \\
\cline{2-5}
%%%%%%%%%%%%%%%%%%%%%%%%%%%%%%%%%%%%%%%%%%%%%%%%%%%%%%
% hueco
& $(1,0,1,1)$
& $(4,4,4)$
& $E_1^{0,2}=E_{\infty}^{0,2}$
& \multirow{8}{*}{$\mathfrak g_4^{\frac {sa}{b}, \frac{\vert b-2\delta \nu\vert}{a}}$} \\
\cline{2-4}
%%%%%%%%%%
%%%%%%%%%%%%%%%%%%%%%%%%%%%%%%%%%%%%%%%%%%%%%%%%%%%%%%
% hueco
& $\begin{array}{c}(1,0,1,b)\\[-4pt] %b\in(0,1)\!\cup\!(1,\infty)
b\in\R^+\!\!-\!\{1\}\end{array}$
& $(4,4,3)$
& $E_1^{0,2}=E_2^{0,2}\neq E_3^{0,2}=E_{\infty}^{0,2}$
&  \\
\cline{2-4}
%%%%%%%%%%
%hueco	
& $\begin{array}{c}(1,1,a,b)\\[-4pt] a>0,\, b \in\R^*, \\[-2pt]
\Theta(\delta,1,a,b)=0\end{array}$
& $(4,4,4)$
& $E_1^{0,2}=E_{\infty}^{0,2}$
&  \\
\cline{2-4}
%%%%%%%%%%
%hueco	
& $\begin{array}{c}(1,1,a,b)\\[-4pt] a>0,\, b \in\R^*, \\[-2pt]
\Theta(\delta,1,a,b)\neq0\end{array}$
& $(4,4,3)$
& $E_1^{0,2}=E_2^{0,2}\neq E_3^{0,2}=E_{\infty}^{0,2}$
&  \\
\hline\hline
%%%%%%%%%%%%%%%%%%%%%%%%%%%%%%%%%%%%%%%%%%%%%%%%%%%%%%%
%	& $\begin{array}{c}(1,0,1,b>0) \text{ or }\\[-4pt] (1,1,a>0,b), b \in\mathbb R\setminus\{0,2\delta\}\end{array}$
%& Change (***)
%& $\mathfrak n_4^{\frac {as}{b}, \frac{\vert b-2\delta \nu\vert}{a}}$ \\
%\cline{2-4}
%%%%%%%%%%%%%%%%%%%%%%%%%%%%%%%%%%%%%%%%%%%%%%%%%%%%%% (1,4,8) %
$(1,4,8)$ 
& $(1,1,0,2\delta)$
& $(4,4,4)$
& $E_1^{0,2}=E_{\infty}^{0,2}$
& $\mathfrak g_5$ \\
\hline\hline
 %%%%%%%%%%%%%%%%%%%%%%%%%%%%%%%%%%%%%%%%%%%%%%%%%%%%% (1,4,6,8) %
% $(1,4,6,8)$ & $\begin{array}{c}(1,0,0,1) \text{ or}\\[-4pt] (1,1,0,b),\, b\neq 0, 2\delta \end{array}$
%& $\begin{array}{l}
%	\omega^1 = \delta\,e^1 - i\,e^2, \\
%	\omega^2 = (\frac{2\,\delta\,\nu}{b} - 1)\,e^3 + i\,\delta\,e^5, \\
%	\omega^3 = (b-2\delta\,\nu) (\delta\,e^6 - i\,e^7), \\
%	\omega^4 = \left(\frac b2-\delta\,\nu\right)\,e^4 - \delta\,\nu\,e^5 + 2\delta(b-2\,\delta\,\nu)\,i\,e^8.
%	\end{array}$
%& $\mathfrak n_6$ \\
%\hline\hline
%%%%%%%%%%%%%%%%%%%%%%%%%%%%%%
\multirow{3}{*}{$(1,4,6,8)$} 
& \multirow{1}{*}{$(1,0,0,1)$}
& \multirow{3}{*}{$(4,3,3)$}
& \multirow{3}{*}{$E_1^{0,2}\neq E_2^{0,2}=E_{\infty}^{0,2}$}
& \multirow{3}{*}{$\mathfrak g_6$} \\[-12pt]
%%%%%%
%hueco
&
&
&
&	\\
\cline{2-2}
%%%%%%%%%%%%%%%%%%%%%%%%%%%%%%
%hueco	
& \multirow{2}{*}{$\begin{array}{c}(1,1,0,b)\\[-3pt] b\neq 0, 2\delta\end{array}$}
& 
&
&	\\[-2pt]
%%%%%
%hueco
&
& 
&
&	\\
\hline\hline
%%%%%%%%%%%%%%%%%%%%%%%%%%%%%%
 %%%%%%%%%%%%%%%%%%%%%%%%%%%%%%%%%%%%%%%%%%%%%%%%%%%%% (1,5,8) %
$(1,5,8)$ 
& $(0,0,0,1)$
& $(4,4,4)$
& $E_1^{0,2}=E_{\infty}^{0,2}$
& $\mathfrak g_7$ \\
\hline\hline
 %%%%%%%%%%%%%%%%%%%%%%%%%%%%%%%%%%%%%%%%%%%%%%%%%%%% (1,5,6,8) %
 $(1,5,6,8)$ 
& $\begin{array}{c}(0,1,0,b)\\[-2pt] b\in\{-1,\, 1\} \end{array}$
& $(4,2,2)$
& $E_1^{0,2}\neq E_2^{2,0}=E_{\infty}^{2,0}$
& $\mathfrak g_8$ \\
\hline
%\multicolumn{4}{l}{\quad} \\[-4pt]
\multicolumn{4}{c}{ \quad\quad\quad \bf{Table 1. FSS for complex structures in Family I} }
\end{tabular}
}
\end{center}

%%%In Table 1,  $\delta=\pm1$, $\Theta(\delta,\nu,a,b)=\big((a-b)^2-2\delta\nu  b\big)\big((a+b)^2-2\delta\nu  b\big)$, and $s=sign(b-2\delta \nu)$.

\newpage

%%%%%%%%%%%%%%%%%%%%%
\renewcommand{\arraystretch}{1.4}
\renewcommand{\tabcolsep}{2pt}
\begin{center}
\begin{tabular}{|c|c|c|c|c|}
%%%%%%%%%%%%%%%%%%%%%%%%%%%%%%%%%%%%%%%%%%%%%%%%%%%% CABECERA %
\hline
\begin{tabular}{c} Ascending \\[-8pt] type \end{tabular} & $J\!\equiv\!(\varepsilon,\mu,\nu,a,b)$ & \small{$(e_1^{0,2},\, e_2^{0,2},\, e_3^{0,2})$} & Sequence $E^{0,2}_r(\frg,\!J)$ & NLA $\frg$ \\
\hline\hline
%%%%%%%%%%%%%%%%%%%%%%%%%%%%%%%%%%%%%%%%%%%%%%%%%%%%% (1,3,5,8) %
\multirow{11}{*}{$(1,3,5,8)$} & $(0,1,0, 0,0)$
& \multirow{2}{*}{$(2,2,2)$}
& \multirow{2}{*}{$E^{0,2}_1\!=E^{0,2}_{\infty}$}
& $\mathfrak g_{9}^0$ \\
\cline{2-2}\cline{5-5}
%%%%%%%%%%%%%%%%%%%%%%%%%%%%%%
	& $(0,1,0, 1, 0)$
& 
&
&$\mathfrak g_{9}^1$ \\
\cline{2-5}
%%%%%%%%%%%%%%%%%%%%%%%%%%%%%%%%%%%%%%%%%%%%%%%%%%%%%%
	& $\begin{array}{c}(1,0,0, a, 0)\\[-4pt] a\in\{0,1\}\end{array}$ 
& \multirow{3}{*}{$(2,2,2)$}
& \multirow{3}{*}{$E^{0,2}_1\!=E^{0,2}_{\infty}$}
&$\mathfrak g_{10}^0$ \\
	\cline{2-2}\cline{5-5}
%%%%%%%%%%%%%%%%%%%%%%%%%%%%%%
	&$\begin{array}{c}(1,0,0, a, b)\\[-4pt] a\!\in\!\{0,1\},\, b\in\R^*\end{array}$  
	& \, &
& $\mathfrak g_{10}^1$ \\
	\cline{2-5}
%%%%%%%%%%%%%%%%%%%%%%%%%%%%%%%%%%%%%%%%%%%%%%%%%%%%%%
	& $(1,1,0, 0, 0)$
& \multirow{5}{*}{$(2,2,1)$} & \multirow{5}{*}{$E^{0,2}_1\!=E^{0,2}_2\!\neq E^{0,2}_3\!=E^{0,2}_{\infty}$}
& $\mathfrak g_{11}^{0,0}$ \\
\cline{2-2}\cline{5-5}
%%%%%%%%%%%%%%%%%%%%%%%%%%%%%%%%%%%%%%%%%%%%%%%%%%%%%%
	& $\begin{array}{c}(1,1,0, a, 0)\\[-4pt] a\in\R^*\end{array}$
& &
& $\mathfrak g_{11}^{1,0}$ \\
\cline{2-2}\cline{5-5}
%%%%%%%%%%%%%%%%%%%%%%%%%%%%%%%%%%%%%%%%%%%%%%%%%%%%%%
	& $\begin{array}{c}(1,1,0, a, b)\\[-4pt] a\in\R,b\in\R^*\end{array}$
& &
& $\mathfrak g_{11}^{\frac{2\sqrt3|a|}{|b|}, 1}$\\
\hline\hline
%%%%%%%%%%%%%%%%%%%%%%%%%%%%%%%%%%%%%%%%%%%%%%%%%%%% (1,3,5,6,8) %
\multirow{3}{*}{$(1,3,5,6,8)$} & $\begin{array}{c}(1,0,1, a, 0)\\[-4pt] a\in\R\end{array}$ 
& \, \multirow{3}{*}{$(2,2,2)$} & \multirow{3}{*}{$E^{0,2}_1\!=E^{0,2}_{\infty}$}
& $\mathfrak g_{12}^0$\\
\cline{2-2}\cline{5-5}
%%%%%%%%%%%%%%%%%%%%%%%%%%%%%%
	& $\begin{array}{c}(1,0,1, a, b)\\[-4pt] a\in\R, b\in\R^*\end{array}$
& \, &
& $\mathfrak g_{12}^1$ \\
	\hline
%\multicolumn{4}{l}{\quad} \\[-4pt]
\multicolumn{4}{c}{ \quad\quad\quad \bf{Table 2. FSS for complex structures in Family II} }
\end{tabular}
\end{center}
%%%%%%%%%%%%%%%%%%%%%

%%%In Table 2, we have that $\varepsilon, \mu, \nu \in \{0,1\}$ with $\mu \nu=0$, and $a,b\in\mathbb R$.

\bigskip

\begin{remark}\label{problem-Stelzig}
{\rm 
In \cite[Problem 11.2]{Stelzig-2021} Stelzig asks for the construction, for every $n\geq 3$, of a compact complex manifold $X$ with $\dim_{\C}X=n$ and with nonvanishing differential on page $E_{n-1}(X)$ starting in bidegree
$(0,n-1)$ or $(0,n-2)$. For $n=3$, any complex nilmanifold $X=(\nilm,J)$ has differential $d_2\colon E_2^{0,1}(X)\longrightarrow E_2^{2,0}(X)$ identically zero (see \cite[Proposition 8.11]{Stelzig-2021} or \cite[Theorem 4.1]{COUV}). 
From the Tables 1 and 2 we get that $d_{3}\colon E_{3}^{0,2}(\frg,J)\longrightarrow E_{3}^{3,0}(\frg,J)$ vanishes for any SnN complex structure $J$ on an 8-dimensional NLA $\frg$. 
One may then ask
whether the differential $d_{n-1}\colon E_{n-1}^{0,n-2}\longrightarrow E_{n-1}^{n-1,0}$ always vanishes on complex nilmanifolds.
}
\end{remark}

We can finally apply the previous results (namely, 
Propositions~\ref{NLA-a-nilvariedad},~\ref{no-deg-fam1-NLA} and~\ref{no-deg-fam2-NLA}) to obtain 
new compact complex manifolds with $d_2\neq0$. To our knowledge, no compact complex manifold of complex dimension $4$ with $d_3\not=0$ is known.
Let $G$ be the simply-connected nilpotent Lie group associated to $\frg$, where $\frg$ is isomorphic to any one of the NLAs 
%%%$\frg_1^{\gamma},\ldots,\frg_{12}^{\gamma}$ 
in the list above. Since $\gamma\in\{0,1\}$, all the Lie algebras are rational except for possibly $\frg_2^{\alpha},\frg_4^{\alpha,\beta}$ and $\frg_{11}^{\alpha,\beta}$. For these algebras, it follows from their structure equations that they are also rational algebras whenever $\alpha,\beta\in\Q$.
Hence, the existence of a lattice $\Gamma$ for the associated nilpotent Lie groups is guaranteed by the well-known Mal'cev theorem \cite{Malcev}. 
Therefore, many compact complex nilmanifolds $X=(\nilm,J)$ 
with $\dim_{\C} X=4$ 
can be defined in this way.

%%%Problem 72. For every $n \geq 3$, construct a $n$-dimensional compact %%%complex manifold $X_n$ with nonvanishing differentials on page $e_r$ %%%starting in degree $(0,n-1)$ or $(0, n - 2)$.
%%
%% Proposition 70. For all left-invariant structures on 6-dimensional nilmanifolds,
%%the differential $d_2^{0,1}\colon E_2^{0,1}\longrightarrow E_2^{2,0}$ %%vanishes.
%%
%% in dim n
%% $d_{n-1}^{0,n-2}\colon E_{n-1}^{0,n-2}\longrightarrow E_{n-1}^{n-1,0}$
%%

\begin{theorem}\label{FSS-on-Xs}
Let $\nilm$ be a nilmanifold endowed with a complex structure $J$ in any of the following cases: 

\hskip-.15cm $\bullet$ $J$ in Family I defined by $(\varepsilon, \nu, a, b)=(1,0,1,b)$ with $b\in\Q^+-\{1\}$, or $(1,1,a,b)$ with $(a,b)\in$ 

\ $\Q^+\times\Q^*$ satisfying $\Theta(\delta,1,a,b)\not=0$;

\vskip.05cm

\hskip-.15cm $\bullet$ $J$ in Family II defined by $(\varepsilon, \mu, \nu, a, b)=(1,1,0,0,0)$,  $(1,1,0,a,0)$ with $a\in\R^*$, or $(1,1,0,a,b)$ 

\ with $(a,b)\in\R\times\R^*$ satisfying $\frac{2\sqrt3|a|}{|b|}\in\Q$.

Then, the compact complex manifold $X=(\nilm,J)$ has Fr\"olicher spectral sequence not degenerating at the second page.
\end{theorem}

\begin{proof}
Let $J$ be a complex structure in Family I defined by $(\varepsilon, \nu, a, b)=(1,0,1,b)$, with $b\in\Q^+-\{1\}$, or $(1,1,a,b)$, with $(a,b)\in\Q^+\times\Q^*$ satisfying $\Theta(\delta,1,a,b)\not=0$. In view of Table 1, in the first case the underlying Lie algebra is $\frg_4^{\frac{1}{b},b}$, and in the second one we have $\frg_4^{\frac{s a}{b},\frac{|b-2\delta|}{a}}$, where $s$ denotes the sign of $b-2\delta$. Since $\delta=\pm1$, we always get rational Lie algebras. Note also that 
$$
\Theta(\delta,1,a,b)=
(a^2-b^2)^2-4\delta b(a^2+b^2)+4 b^2,
%%%=\left( a^2-b^2-2\delta b-4b\sqrt{2\delta b} \right)\left( a^2-b^2-2\delta b+4b\sqrt{2\delta b} \right)
$$
hence, given any $b\in\Q^*$ we can choose $\delta$ so that $\delta b<0$ and thus 
$\Theta(\delta,1,a,b)\not=0$. 

\smallskip

Now, let $J$ be a complex structure in Family II defined by $(\varepsilon, \mu, \nu, a, b)=(1,1,0,0,0)$,  $(1,1,0,a,0)$ with $a\in\R^*$, or $(1,1,0,a,b)$ with $(a,b)\in\R\times\R^*$ satisfying $\frac{2\sqrt3|a|}{|b|}\in\Q$. 
In view of Table 2, the underlying NLAs are $\frg_{11}^{0,0}$, $\frg_{11}^{1,0}$ or $\frg_{11}^{q,1}$, with $q=\frac{2\sqrt3|a|}{|b|}\in\Q$ and $q\geq 0$, depending on the case, which are always rational. Note that for the last two NLAs there are one-parameter families of (non-isomorphic) complex structures.

\smallskip

To get the desired result it suffices to observe that for any lattice $\Gamma$ in the Lie group $G$ associated to any of these rational Lie algebras, the FSS satisfies $E_2(\nilm,J)\neq E_3(\nilm,J)$ by Propositions~\ref{NLA-a-nilvariedad}, \ref{no-deg-fam1-NLA} and~\ref{no-deg-fam2-NLA}.
\end{proof}

%%%%%%%%%%%%%%%%%%%%%%%%%%%%%%%
\section{The Fr\"olicher spectral sequence of balanced manifolds}\label{balanced-section}
%%%%%%%%%%%%%%%%%%%%%%%%%%%%%%%

\noindent 
In this section we study the existence of balanced metrics on $8$-dimensional nilmanifolds $M$ 
endowed with an SnN complex structure~$J$. 
As an application, we get infinitely many compact balanced 
manifolds with complex dimension 4 and different homotopy types whose FSS does not degenerate at the second page. This is also extended to non-degeneration at any arbitrary page.
%%%We recall that these pairs $(\frg,J)$ are parametrized,
%%%up to equivalence, by Family I and Family II in Section~\ref{FSS-section}.

\medskip
A compact complex manifold $X$ of complex dimension $n$ is said to be \emph{balanced} if there is a Hermitian metric $F$ on $X$ 
satisfying $dF^{n-1}=0$. These metrics were first studied in \cite{Mi} and play an important role in geometry and many aspects in theoretical and mathematical physics.
Notice that for any $n\geq 3$ one has
$$dF^{n-1}=(n-1)\,dF\wedge F^{n-2}
	=(n-1)\,\big(\partial F \wedge F^{n-2} + \bar{\partial} F\wedge F^{n-2}\big).$$
Since $F$ is a real form, the two summands in the last expression above are conjugate to each other, 
and the balanced condition is equivalent to
\begin{equation}\label{balanced-equiv}
\partial F \wedge F^{n-2}=0.
\end{equation}

We recall that when $X=(M,J)$ is a nilmanifold $M$ 
endowed with an invariant complex structure~$J$, 
%%%Let $M=\nilm$ be a nilmanifold, i.e. a compact quotient of a connected and simply connected nilpotent Lie group $G$ by a lattice $\Gamma$, endowed with a left-invariant complex structure $J$ (namely, defined on the Lie algebra $\frg$ of $G$).
%%%By a result of Belgun~\cite{Bel}, 
by the symmetrization process the existence of balanced metric on $X$ implies the
existence of an invariant one, i.e. a balanced metric on the underlying Lie algebra $(\frg, J)$.

Let us consider a basis of invariant $(1,0)$-forms $\{\omega^k\}_{k=1}^4$ on a complex nilmanifold $X=(M,J)$, where $M$ is $8$-dimensional. Then, in terms of this basis, any invariant Hermitian metric $F$ on $X$ is given by 
\begin{equation}\label{formaFund}
F=\sum_{k=1}^4 i\,x_{k\bar k}\,\omega^{k\bar k}
   \ +\sum_{1\leq k<l\leq 4}\big( x_{k\bar l}\,\omega^{k\bar l}-\bar x_{k\bar l}\,\omega^{l\bar k} \big),
\end{equation}
for some coefficients $x_{k\bar k}\in\mathbb R$ and $x_{k\bar l}\in\mathbb C$. 
Associated to $F$, we consider the $4\times 4$ matrix 
\begin{equation}\label{matriz}
H=\begin{pmatrix}
x_{1\bar1} & -i\,x_{1\bar 2} & -i\,x_{1\bar 3} & -i\,x_{1\bar 4} \\
i\,\bar x_{1\bar2} & x_{2\bar2} & -i\,x_{2\bar 3} & -i\,x_{2\bar 4} \\
i\,\bar x_{1\bar3} & i\,\bar x_{2\bar3} & x_{3\bar 3} & -i\,x_{3\bar 4} \\
i\,\bar x_{1\bar4} & i\,\bar x_{2\bar4} & i\,\bar x_{3\bar4} & x_{4\bar 4}
\end{pmatrix}.
\end{equation}
Let us denote by $H_{rs}$ the determinant of the $3\times 3$ submatrix
obtained by removing the $r$-th row and the $s$-th column from $H$. Note that $H_{sr}=\bar{H}_{rs}$. The metric $F$ being positive implies, in particular, $H_{rr}>0$ for every $1\leq r\leq 4$.

\vskip.2cm

\begin{lemma}\label{balanced-FI-FII}
Let $M$ be an $8$-dimensional nilmanifold endowed with an SnN complex structure~$J$. Let $F$ be any Hermitian metric on $(M,J)$ defined by~\eqref{formaFund} in a $(1,0)$-frame $\{\omega^k\}_{k=1}^4$ satisfying 
the equations~\eqref{ecus-I} or~\eqref{ecus-II}. We have:
\begin{enumerate}
\item[(i)] If $J$ belongs to Family I, then $F$ is balanced if and only if 
%%\begin{equation}\label{lema-balanced-I}
$$
\varepsilon=\nu=0, 
\qquad 
a\, H_{12} + \bar{H}_{14}=0, 
\qquad
b\, H_{22} - 2\delta\,\Imag\!(H_{13})=0.
$$
%%%\end{equation}
%
\item[(ii)] If $J$ belongs to Family II, then $F$ is balanced if and only if 
$$
\nu=0, \quad
H_{14}=0, \quad
2\varepsilon\,\Imag\!(H_{12})-\mu\,\bar{H}_{24} - ia\, H_{11}=0,
\quad
\mu\,H_{22} - 2b\,\Imag\!(H_{12}) + 2\,\Imag\!(H_{13}) = 0.$$
\end{enumerate}
\end{lemma}

\begin{proof}
To prove (i) we calculate \eqref{balanced-equiv} for $n=4$ taking into account the complex equations~\eqref{ecus-I}. We get
%%%\begin{equation}\label{cond-familias}
$$
0=\frac12\, \partial F \wedge F^{2}
= A_{I}\,\omega^{1234\bar 1\bar2\bar 3}
+B_{I}\,\omega^{1234\bar 1\bar2\bar 4}
+C_{I}\,\omega^{1234\bar 1\bar3\bar 4},
$$
%%%\end{equation}
where 
\begin{equation*}
\begin{split}
A_{I}&=-\nu\,H_{11} -i\, b\,H_{22} + 2\,i\,\delta\,\Imag\!(H_{13}), \\
B_{I}&=-\delta\,\varepsilon\,b\,\bar{H}_{12} - i\, a\,H_{12} -i\,\bar{H}_{14}, \\
C_{I}&=-i\,\varepsilon\,H_{11}.
\end{split}
\end{equation*}
Since $H_{11}>0$, the condition $C_{I}=0$ implies $\varepsilon=0$. 
Similarly, since $\Real(A_{I})=-\nu\,H_{11}=0$ we get $\nu=0$.
Now, taking $\varepsilon=\nu=0$, the remaining conditions 
are $\Imag(A_{I})=0$ and $B_{I}=0$, and the statement in the part (i) of the lemma follows.

\smallskip

For the proof of (ii) we use the complex equations~\eqref{ecus-II} to compute \eqref{balanced-equiv} for $n=4$. One has
$$
0=\frac12\,\partial F \wedge F^{2}
= A_{I\!I}\,\omega^{1234\bar 1\bar2\bar 3}
+B_{I\!I}\,\omega^{1234\bar 1\bar2\bar 4}
+C_{I\!I}\,\omega^{1234\bar 1\bar3\bar 4},
$$
where 
\begin{equation*}
\begin{split}
A_{I\!I}&=-\nu\,H_{11} + i\,\Big(\mu\,H_{22} - 2\,b\,\Imag\!(H_{12})
	+ 2\,\Imag\!(H_{13})\Big), \\
B_{I\!I}&=-2\,\varepsilon\,\Imag\!(H_{12}) + \mu\,\bar{H}_{24} + i\,a\, H_{11}, \\
C_{I\!I}&=i\,\bar{H}_{14}.
\end{split}
\end{equation*}
From $\Real(A_{I\!I})=-\nu\,H_{11}=0$ and $H_{11}>0$, it follows that $\nu=0$. The remaining conditions in the statement come
directly from $\Imag(A_{I\!I})=0$, $B_{I\!I}=0$ and $C_{I\!I}=0$.
\end{proof}

The following result provides a classification of the SnN complex structures $J$ in eight dimensions admitting balanced metric. 

\begin{proposition}\label{prop-bal-fam1-fam2}
Let $M$ be an $8$-dimensional nilmanifold endowed with an SnN complex structure~$J$. Then, $(M,J)$ admits a balanced metric if and only if 
$J$ is equivalent to a complex structure defined by, either~\eqref{ecus-I} with tuple $(\varepsilon, \nu, a, b)$ being one of the following 
$$(0,0,0,1),\, (0,0,1, 0),\, (0,0, 1, 1),$$
or~\eqref{ecus-II} with tuple $(\varepsilon,\mu,\nu, a, b)$ one of the following 
$$(1, 1, 0, a, b),\, (1, 0, 0, 0, b),\, (0, 1, 0, 0, 0), \, (0, 1, 0, 1, 0).$$
\end{proposition}

\begin{proof}
When $J$ belongs to Family I, it follows from Lemma~\ref{balanced-FI-FII}~(i)
that the possible values for the tuple 
$(\varepsilon, \nu, a, b)$ are $(0,0,0,1)$, $(0,0,1, 0)$ or $(0,0, 1, 1)$. 
We will find a balanced metric $F$ for each case. 
Observe that $a,\,b\in\{0,1\}$. We distinguish
two cases according to the value of $b$:

$\bullet$ If $b=0$, then $a=1$ and the equations in Lemma~\ref{balanced-FI-FII}~(i) reduce to 
$H_{12} + \bar{H}_{14}=0$  
and 
$\Imag\!(H_{13})=0$. To obtain a solution it suffices to consider the metric $F$ defined by taking $H$ in~\eqref{matriz} as the identity matrix.
%%%$$x_{k\bar k}=1, \quad x_{k\bar l}=0, \text{ \ for }1\leq k,l\leq 4.$$

$\bullet$ If $b=1$, we consider the metric $F$ defined by 
$$x_{1\bar1}=x_{3\bar3}=x_{4\bar4}=1, \quad x_{2\bar2}=4, \quad
x_{1\bar 2}=i, \quad
x_{1\bar 3}=x_{1\bar4}=x_{3\bar 4}=0,  \quad x_{2\bar 3}=\frac{\delta}{2}, \quad x_{2\bar 4}=i\,a.$$
It can be checked that this choice indeed defines a positive-definite metric for which one has  $H_{12}=1$, $H_{13}=\frac{\delta\, i}{2}$, $H_{14}=-a$, and $H_{22}=1$. 
Since $\delta=\pm1$, all the conditions in Lemma~\ref{balanced-FI-FII}~(i) are satisfied and the metric $F$ is balanced.

\smallskip

Suppose now that $J$ belongs to Family II. Lemma~\ref{balanced-FI-FII}~(ii)
implies that the possible values for the tuple 
$(\varepsilon, \mu, \nu, a, b)$ are 
$(1, 1, 0, a, b)$, $(1, 0, 0, 0, b)$, $(1, 0, 0, 1, b)$, $(0, 1, 0, 0, 0)$ or $(0, 1, 0, 1, 0)$. 
Let us distinguish
two cases depending on the value of the parameter $\mu$:

$\bullet$ If $\mu=0$, then we have $(\varepsilon, \mu, \nu, a, b)=(1, 0, 0, 0, b)$ or $(1, 0, 0, 1, b)$.
Observe that one of the equations in 
Lemma~\ref{balanced-FI-FII}~(ii) becomes $2\,\Imag\!(H_{12})- i\,a\, H_{11}=0$, which is equivalent to $\Imag\!(H_{12})=0$ and $a\, H_{11}=0$. 
Since $H_{11}>0$, we conclude that $a=0$ in order that $J$ admits a balanced metric.
Therefore, there are no balanced metrics for 
the tuple $(1, 0, 0, 1, b)$. However, for the case $(1, 0, 0, 0, b)$ one can check that the metric $F$ defined by taking $H$ in~\eqref{matriz} as the identity matrix is balanced.
 
$\bullet$  If $\mu=1$, we take $F$ defined by 
$$x_{1\bar1}=x_{3\bar3}=x_{4\bar4}=1, \quad x_{2\bar2}=a^2+\frac 34, \quad
x_{1\bar 2}=x_{1\bar 4}=x_{2\bar3}=x_{3\bar 4}=0, \quad
x_{1\bar 3}=\frac 12, \quad x_{2\bar 4}=-a.$$
It can be checked that it defines a positive-definite metric. Moreover, $H_{11}=\frac{3}{4}$, $H_{12}=0$, $H_{13}=-\frac{3\,i}{8}$, $H_{14}=0$, $H_{22}=\frac{3}{4}$, and $H_{24}=\frac{3\,a\,i}{4}$. 
Since all the conditions in Lemma~\ref{balanced-FI-FII}~(ii) are satisfied, the metric $F$ is balanced.
\end{proof}

The following result is a consequence of Proposition~\ref{prop-bal-fam1-fam2} and the values given in the second and fifth columns of Tables 1 and 2. 
Here we are also using that the existence of a balanced metric implies the existence of an invariant one. 

\begin{theorem}\label{theorem-balanced}
If an $8$-dimensional nilmanifold endowed with an SnN complex structure admits a balanced metric, then its underlying Lie algebra is isomorphic to 
$\frg_1^{\gamma}$, $\frg_7$, $\frg_9^{\gamma}$, $\frg_{10}^{\gamma}$ or $\frg_{11}^{\alpha,\beta}$.
\end{theorem}

This result has a certain converse which is useful in the construction of balanced nilmanifolds satisfying additional properties. 
Indeed, by a similar argument as in Section~\ref{FSS-section}, starting with the rational Lie algebras $\frg_1^{\gamma}$, $\frg_7$ or~$\frg_9^{\gamma}$ we get  compact balanced nilmanifolds. Similarly, when we start with $\frg_{10}^{\gamma}$ endowed with any complex structure corresponding to a tuple $(\varepsilon, \mu, \nu, a, b)=(1, 0, 0, 0, b)$. 
Finally, compact balanced nilmanifolds can also be constructed starting with $\frg_{11}^{\alpha,\beta}$ for $ (\alpha, \beta)=(0,0)$, $(1,0)$ or 
$(q, 1)$ with rational $q\geq 0$, since from Table 2 in the latter case we can always choose complex structures defined by tuples $(\varepsilon, \mu, \nu, a, b)=(1,1,0, a, b)$, with $b\not=0$, such that~$\frac{2\sqrt3|a|}{|b|}=q$.

We observe that all the compact balanced nilmanifolds $X$ with underlying Lie algebra isomorphic to $\frg_{11}^{\alpha,\beta}$ satisfy $E_2(X) \not= E_{\infty}(X)$, due to Proposition~\ref{no-deg-fam2-NLA}. Moreover, non-isomorphic Lie algebras $\frg_{11}^{\alpha,\beta}$ and $\frg_{11}^{\alpha',\beta'}$ give rise to nilmanifolds $X$ and $X'$ with different minimal model by Hasegawa theorem \cite{Hasegawa}, hence with different real homotopy type (see \cite{BM2012, DGMS, FHT-libro, GM-libro, Sullivan} for results in homotopy theory). 
%%%Furthermore, the complex homotopy type of $X$ and $X'$ is different, too (see Remark~\ref{iso-complejo}). In consequence, we have:
Furthermore, their complex homotopy types are also different:

\begin{theorem}\label{cor-infinite-bal-FSS}
There are infinitely many complex (hence, real or rational) homotopy types of compact 
balanced manifolds of complex dimension $4$ with Fr\"olicher spectral sequence not degenerating at the second page.
\end{theorem}

\begin{proof}
By the discussion above, it is enough to consider a family of compact balanced nilmanifolds $Y^4_{\alpha}$ with underlying Lie algebra isomorphic to $\frg_{11}^{\alpha,\beta}$ with $\alpha=q\geq 0$, $q\in \Q$, and $\beta=1$. 
To complete the proof it only remains to prove that the Lie algebras $\frg_{11}^{\alpha,1}$ and $\frg_{11}^{\alpha',1}$ are not isomorphic over $\C$ whenever $\alpha\neq \alpha'$. Indeed, this means that the $\C$-minimal models of  $Y^4_{\alpha}$ and  $Y^4_{\alpha'}$ are not isomorphic for $\alpha\neq \alpha'$. 
The proof of this fact is quite long and technical, and we give the details in the Appendix.
%but somehow similar to the one given in \cite[Lemmas 3.4, 3.5 and 3.6]{LUV2} for the real case, so we omit here the details.
\end{proof}

%%%\vskip.2cm
%%%\begin{remark}\label{iso-complejo}
%%%{\rm One can check that the proofs given in \cite[Lemmas 3.4, 3.5 and 3.6]{LUV2} are still valid to assert that the Lie algebras $\frg_{11}^{\alpha,1}$ and $\frg_{11}^{\alpha',1}$ are not isomorphic over $\C$ for any $\alpha,\alpha'\in [0,\infty)$ with $\alpha\neq \alpha'$. }
%%%\end{remark}

Our next goal is to extend this result to get non-degeneration %so that the non degeneration occurs 
in arbitrary high pages. 
For this we first consider the nilmanifolds constructed by Bigalke and Rollenske in \cite[Theorem 1]{BR}. %they construct,  
For every $n\geq 2$, 
they provide
a complex nilmanifold $X^{4n-2}$ 
of complex dimension $4n-2$ with %, such that its 
Fr\"olicher spectral sequence %does not degenerate 
not degenerating
at the $E_n$ term, i.e., $d_n \not= 0$. 
More concretely, the nilmanifold $X^{4n-2}$ has a basis of invariant (1,0)-forms given by 
$$
dx_1,\ldots,dx_{n-1},dy_1,\ldots,dy_{n},dz_1,\ldots,dz_{n-1},\omega^1,\ldots,\omega^n
$$ 
and satisfying the complex structure equations
\begin{equation}\label{ecusBR-1}
\left\{\begin{array}{ll}
d(dx_i)=d(dz_i)=0, &  i=1,\ldots,n-1,\\[1pt]
d(dy_j)=0, & j=1,\ldots,n,\\[1pt]
d\omega^1= - d\bar{y}^{1}\wedge dz_1, & \\[1pt]
d\omega^{k}=dx_{k-1}\wedge dy_k + dy_{k-1}\wedge d\bar{z}_{k-1}, & k=2,\ldots,n.
\end{array}\right.
\end{equation}
Let us consider the (1,0)-frame $\{\tau^i\}$ defined as

\vskip.3cm

$\tau^i=dx_i \ (1\leq i\leq n-1)$, \ $\tau^i=dy_{i-n+1} \ (n\leq i\leq 2n-1)$,  \ $\tau^i=dz_{i-2n+1} \  (2n\leq i\leq 3n-2)$,

\vskip.3cm

$\tau^i=\omega^{i-3n+3} \ (3n-1\leq i\leq 4n-3)$, and \ $\tau^{4n-2}=\omega^1$.

\vskip.3cm

\noindent The complex structure equations~\eqref{ecusBR-1} are expressed in this frame as
\begin{equation}\label{ecusBR-2}
\left\{\begin{array}{l}
d\tau^1=d\tau^2=\cdots=d\tau^{n-1}=0, \\[2pt]
d\tau^n=d\tau^{n+1}=\cdots=d\tau^{2n-2}=d\tau^{2n-1}=0, \\[2pt]
d\tau^{2n}=d\tau^{2n+1}=\cdots=d\tau^{3n-2}=0, \\[2pt]
d\tau^{3n-1}=\tau^1\wedge \tau^{n+1} \quad\ \,  + \tau^{n}\wedge \overline{\tau^{2n}}, \\[-2pt]
\quad \vdots \qquad\qquad\ \  \vdots \qquad\qquad\qquad\ \ \vdots  \\[-1pt]
d\tau^{4n-3}=\tau^{n-1}\wedge \tau^{2n-1} + \tau^{2n-2}\wedge \overline{\tau^{3n-2}}, \\[2pt]
d\tau^{4n-2}= \qquad\qquad\qquad\quad   \tau^{2n}\wedge \overline{\tau^{n}}.
\end{array}\right.
\end{equation}

Note that the complex structure $J$ on the (2-step) nilmanifold $X^{4n-2}$ is of nilpotent type.
Bigalke and Rollenske prove the following result:

\begin{proposition}\label{BR-dn}\cite[Lemma 2]{BR}
The differential form $\beta = \overline{\tau^{2n+1}} \wedge \ldots \wedge \overline{\tau^{3n-2}} \wedge  \overline{\tau^{4n-2}}$ defines a class $[\beta] \in E_n^{0,n-1}(X)$ such that   $d_n([\beta])= [\tau^{1} \wedge \ldots \wedge \tau^{n-1} \wedge \tau^{2n-1}] \not=0$ in $E_n^{n,0}(X)$.
%%%\footnote{Luis: en el artículo prueban que la forma $\beta = \bar{\omega}^1 \wedge  d\bar{z}_{2} \wedge \ldots \wedge d\bar{z}_{n-1}$ define una clase $[\beta] \in E_n^{0,n-1}(X)$ tal que $d_n([\beta])=(-1)^{n-2} [dx_{1} \wedge \ldots \wedge dx_{n-1}\wedge dy_{n}] \not=0$ en $E_n^{n,0}(X)$. Al expresarlo en la base $\tau$ queda lo que aparece ahora escrito. Confirmar! Adela: comprobado}
\end{proposition}

As shown by Sferruzza and Tardini in \cite{ST21}, the nilmanifold $X^{4n-2}$ is balanced. Indeed, many balanced metrics exist on $X^{4n-2}$. For  any $\rho=(\rho_1,\ldots,\rho_{4n-2})\in (\R^+)^{4n-2}$, let $F_{\rho}$ be the ``diagonal'' (with respect to the above basis) invariant Hermitian metric defined as
$$
F_{\rho}=\frac{i}{2} \left(\rho_1 \,\tau^{1}\wedge \overline{\tau^{1}}+\cdots+\rho_{4n-2} \,\tau^{4n-2}\wedge \overline{\tau^{4n-2}} \right).
$$
Hence, 
\begin{equation}\label{partial-BR}
\partial F_{\rho}^{4n-3} 
= \sum_{k=1}^{4n-2} h_k^{\rho}\ \partial\!\left(\!(\tau^{1}\wedge \overline{\tau^{1}}) \wedge\cdots\wedge  (\hat{\tau^{k}} \wedge \hat{\overline{\tau^{k}}})  \wedge\cdots\wedge(\tau^{4n-2}\wedge \overline{\tau^{4n-2}})\!\right),
\end{equation}
where $h_k^{\rho}= i\, \frac{(4n-3)!}{2^{4n-3}}\, \rho_1\cdots \hat{\rho_{k}} \cdots \rho_{4n-2}$. Here, $\hat{\rho_{k}}$ means that we are removing $\rho_{k}$ from the expression, and a similar meaning is given to the notation 
$\hat{\tau^{k}}$ and $\hat{\overline{\tau^{k}}}$.
%mean that we are removing the forms $\tau^{k}$, $\overline{\tau^{k}}$ from the expression.
Using \eqref{ecusBR-2} one has $\partial (\tau^{j}\wedge \overline{\tau^{j}})=0$ for $1\leq j\leq 3n-2$, 
and 
$$
\begin{array}{llll}
\partial (\tau^{3n-1}\wedge \overline{\tau^{3n-1}})=& \tau^1\wedge \tau^{n+1}\wedge \overline{\tau^{3n-1}}&+& \tau^{2n}\wedge \overline{\tau^{n}}\wedge\tau^{3n-1},\\[3pt]
\qquad\qquad \vdots&\qquad\qquad \vdots & & \qquad\quad \vdots \\[3pt]
\partial (\tau^{4n-3}\wedge \overline{\tau^{4n-3}})=& \tau^{n-1}\wedge \tau^{2n-1}\wedge \overline{\tau^{4n-3}} &+& \tau^{3n-2}\wedge \overline{\tau^{2n-2}}\wedge \tau^{4n-3}, \\[5pt]
\partial (\tau^{4n-2}\wedge \overline{\tau^{4n-2}})=& & & \tau^{n}\wedge \overline{\tau^{2n}}\wedge\tau^{4n-2}.
\end{array}
$$
These equalities clearly imply that 
$\xi\wedge \partial (\tau^{j}\wedge \overline{\tau^{j}})=0$ 
for every $\xi \in \bigwedge^{3n-3}\langle \tau^{1}\wedge \overline{\tau^{1}},\ldots,\tau^{3n-2}\wedge \overline{\tau^{3n-2}} \rangle$ and for any $1\leq j\leq 4n-2$,
so \eqref{partial-BR} vanishes. Thus, any diagonal metric $F_{\rho}$ is balanced.

\begin{proposition}\label{BR-balanced}\cite[Theorem 3.3]{ST21}
For every $n\geq 2$, the nilmanifold $X^{4n-2}$ is balanced. 
\end{proposition}

In a recent paper Stelzig proves that the K\"unneth formula is compatible with the Hodge filtration and its conjugate, so
the well-known K\"unneth formula for the Dolbeault cohomology implies a K\"unneth formula for
all higher pages of the Fr\"olicher spectral sequence (see \cite[Section 6]{Stelzig-2021} and the proof of \cite[Proposition 4.8]{Stelzig-2021} for more details).

\begin{proposition}\label{kunneth}\cite{Stelzig-2021}
Let $X$ and $Y$ be compact complex manifolds, and let $Z=X\times Y$. For every $r\geq 1$ and every $(p,q)$ with $0\leq p,q\leq \dim_{\C}Z$, the following equality holds
\vskip-.3cm
\hskip4.5cm$
%%%\begin{equation}\label{kunneth}
e_r^{p,q}(Z)=\!\!\!\!{{\begin{array}{cc}& \phantom{-} \\[-4pt]& \sum\\[-8pt]& _{p_1+p_2=p}\\[-10pt] & _{q_1+q_2=q}\end{array}}} e_r^{p_1,q_1}(X) \cdot e_r^{p_2,q_2}(Y),
%%%\end{equation}
$
\vskip.2cm
\noindent where $e_r^{p,q}(\cdot)$ denotes the dimension of $E_r^{p,q}(\cdot)$, i.e.  $e_r^{p,q}(\cdot)=\dim E_r^{p,q}(\cdot)$.
\end{proposition}

We will apply this formula together with the results above 
%%%Theorem~\ref{cor-infinite-bal-FSS} and Proposition~\ref{BR-balanced}
to prove the following 

\begin{theorem}\label{th-inf-balanced}
For every integer $n\geq 3$, 
there exist infinitely many compact balanced 
manifolds~$Z$ with $\dim_{\C}Z=4n+2$ and different complex (hence, real or rational) homotopy types whose Fr\"olicher spectral sequence $\{E_r(Z)\}$ does not degenerate at the $n$-th page. 

Similar results hold for the first and second pages when complex dimensions $7$ and $4$ are respectively considered.
%A similar result holds for $n=1$ and $n=2$ in complex dimension $7$ and $4$, respectively.
\end{theorem}

\begin{proof} 
First, we prove the theorem for each $n\geq 3$. Let $X^{4n-2}$ be the nilmanifold in the Bigalke-Rollenske family, which is balanced by Proposition~\ref{BR-dn}. Let us also consider any of the complex balanced manifolds $Y^4_{\alpha}$ in the  infinite family given in Theorem~\ref{cor-infinite-bal-FSS}. 
%Since $X^{4n-2}$ and $Y^4_{\alpha}$ are balanced, 
Then, by \cite[Proposition 1.9]{Mi} the compact complex manifold $Z^{4n+2}_{\alpha}=X^{4n-2}\times Y^4_{\alpha}$ is balanced too. 

Now observe that both $X^{4n-2}$ and $Y^4_{\alpha}$ are nilmanifolds, and
recall that $Y^4_{\alpha}$ is constructed from a nilpotent Lie algebra $\frg_{11}^{\alpha,1}$, with $\alpha\geq 0$ and $\alpha\in\Q$. By \cite{Hasegawa} the ($\Q$-)minimal model of $Z^{4n+2}_{\alpha}$ is isomorphic to the commutative differential graded algebra (CDGA) defined by the exterior algebra of the rational Lie algebra 
$\frh\oplus \frg_{11}^{\alpha,1}$, where $\frh$ denotes the nilpotent Lie algebra underlying the Bigalke-Rollenske nilmanifold $X^{4n-2}$. Hence, for every non-negative $\alpha,\,\alpha'\in\mathbb Q$ with $\alpha\neq \alpha'$,
the minimal models of 
$Z^{4n+2}_{\alpha}$ and $Z^{4n+2}_{\alpha'}$ are isomorphic to 
$\frg^{\alpha}=\frh\oplus \frg_{11}^{\alpha,1}$ and $\frg^{\alpha'}=\frh\oplus \frg_{11}^{\alpha',1}$, respectively. Let us denote their complexifications by $\frg^{\alpha}_{_{\C}}$ and $\frg^{\alpha'}_{_{\C}}$. By \cite[Lemma 1]{Seeley}, if a finite-dimensional Lie algebra over $\C$ is decomposed as a direct sum of indecomposable ideals, then the isomorphism classes of these ideals are unique.
%%%\footnote{Adela: He reescrito esta parte. Es correcto as\'i? Tambi\'en he introducido lo de $\alpha\neq \alpha'$ ya que antes no se dec\'ia}
%
%in any decomposition of a finite-dimensional Lie algebra over $\C$ as a direct sum of indecomposable ideals, the isomorphism classes of the ideals are unique. 
%Let $L$ be any finite-dimensional Lie algebra over $\C$. 
%By \cite[Lemma 1]{Seeley}, in any decomposition of $L$ as a direct sum of indecomposable ideals, the isomorphism classes of the ideals are unique. 
%Denote by $\frg^{\alpha}_{\C}$ and $\frg^{\alpha'}_{\C}$ the complexification of the Lie algebras $\frg^{\alpha}$ and $\frg^{\alpha'}$, respectively. 
Hence, if we assume that $\frh_{_{\C}}\oplus (\frg_{11}^{\alpha,1})_{_{\C}}=\frg^{\alpha}_{_{\C}} \cong \frg^{\alpha'}_{_{\C}}=\frh_{_{\C}}\oplus (\frg_{11}^{\alpha',1})_{_{\C}}$, 
then $(\frg_{11}^{\alpha,1})_{_{\C}} \cong (\frg_{11}^{\alpha',1})_{_{\C}}$ and the 
nilpotent Lie algebras $\frg_{11}^{\alpha,1}$ and $\frg_{11}^{\alpha',1}$ would be isomorphic over $\C$. However, this is a contradiction (see %%%Remark~\ref{iso-complejo} 
the proof of Theorem~\ref{cor-infinite-bal-FSS}).
%%%\footnote{El resultado de Seeley es sobre $\C$ y lo podemos usar porque sabemos que las álgebra complexificadas $(\frg_{11}^{\alpha',1})_{\C}$ son no isomorfas por la demostraci\'on del  Theorem~\ref{cor-infinite-bal-FSS}.}
Consequently, the (nil)manifolds $Z^{4n+2}_{\alpha}$ and $Z^{4n+2}_{\alpha'}$  have different complex homotopy type for $\alpha\neq \alpha'$, thus also different real or rational homotopy type.
%%%
\begin{comment}
Since $\frg_{11}^{\alpha,1}$ and $\frg_{11}^{\alpha',1}$ are not isomorphic, if $F\colon \frh\oplus \frg_{11}^{\alpha,1} \longrightarrow \frh\oplus \frg_{11}^{\alpha',1}$ is an isomorphism of Lie algebras, then the subalgebra $\frg_{11}^{\alpha,1}$ would be isomorphic to a 
subalgebra $\frk$ in $\frh\oplus \frg_{11}^{\alpha',1}$ which is not contained in $\frg_{11}^{\alpha',1}$. The nilpotency implies that the dimension of the center of $\frk$ is at least 2, which is a contradiction to the fact that $\frg_{11}^{\alpha,1}$ has 1-dimensional center. So, there is no isomorphism $F$ and the manifolds  $X\times Y$ and $X\times Y'$ have different rational homotopy type.
\end{comment}
%%%

%%%Hence, for any complex manifold $Y$, the K\"unneth formula in Proposition~\ref{kunneth} for $Z=X\times Y$ with $p=0$ and $q=n-1$ gives
%%%\begin{equation}\label{kunneth-2}
%e_r^{0,n-1}(Z)=\sum_{l=0}^{\dim_{\C}Y} e_r^{0,n-1-l}(X) \cdot e_r^{0,l}(Y).
%%%\end{equation}

Let us now apply the K\"unneth formula in Proposition~\ref{kunneth} to $Z^{4n+2}_{\alpha}=X^{4n-2}\times Y^4_{\alpha}$ with $p=0$, $q=n-1$ and $r\in\{n,n+1\}$.  Then, we get:
$$
%\begin{equation}\label{kunneth-3}
\begin{array}{rl}
e_n^{0,n-1}(Z^{4n+2}_{\alpha})
%%%=&\sum_{l=0}^{4} e_m^{0,m-1-l}(X) \cdot e_m^{0,l}(Y)\\
=\!\!&\!\! e_n^{0,n-1}(X^{4n-2}) \cdot e_n^{0,0}(Y^4_{\alpha})+\sum_{l=1}^{4} e_n^{0,n-1-l}(X^{4n-2}) \cdot e_n^{0,l}(Y^4_{\alpha})\\
>\!\!&\!\! e_{n+1}^{0,n-1}(X^{4n-2}) \cdot e_{n+1}^{0,0}(Y^4_{\alpha})+ \sum_{l=1}^{4} e_{n+1}^{0,n-1-l}(X^{4n-2}) \cdot e_{n+1}^{0,l}(Y^4_{\alpha})\ = e_{n+1}^{0,n-1}(Z^{4n+2}_{\alpha}). 
%%%\\
%%%=&\sum_{l=0}^{4} e_{m+1}^{0,m-1-l}(X) \cdot e_{m+1}^{0,l}(Y)\\
%%%=&e_{m+1}^{0,m-1}(Z).
\end{array}
%\end{equation}
$$
Here we have used that $e_{n}^{0,n-1}(X^{4n-2})>e_{n+1}^{0,n-1}(X^{4n-2})$ by Proposition~\ref{BR-dn}, 
%$e_r^{0,0}(Y)=\dim H_{\db}^{0,0}(Y)=1$ for every $r\geq1$, 
together with the well-known Fr\"olicher inequalities $e_r^{p,q}(\cdot)\geq e_{r+1}^{p,q}(\cdot)$, valid for every compact complex manifold and any~$r,p,q$.
In conclusion, for every integer $n\geq 2$, the infinite family of compact balanced manifolds $Z^{4n+2}_{\alpha}$ have differential $d_n\not=0$. 

\smallskip

The result for the second page comes directly from Theorem~\ref{cor-infinite-bal-FSS}. Therefore, we next focus on the non-degeneration at the first page. Note that we can no longer make use of the Bigalke-Rollenske nilmanifolds since they satisfy $d_n\not=0$ with $n\geq 2$.
%%%\footnote{Adela: no entiendo bien a qu\'e se refiere esta frase. Quiz\'a a que podr\'ian verificar $E_1\cong\cdots\cong E_n\ncong E_{n+1}$?} 
Moreover, we do not know if the nilmanifolds $Y^4_{\alpha}$ have non-zero differential $d_1$. 
%%%.\footnote{Luis: es probable que $E_1=E_2$ en esta familia y entonces $d_1$ sería siempre nulo, por lo que tomo el producto por otra balanced con $d_1\not=0$.} 
So, we will proceed as follows.

%It remains to prove the case $n=1$, since $n=2$ is given in Theorem~\ref{cor-infinite-bal-FSS}. Note that the Bigalke-Rollenske nilmanifolds are defined for $n\geq 2$. On the other hand, we do not know if the nilmanifolds $Y^4_{\alpha}$ have non-zero differential $d_1$.\footnote{Luis: es probable que $E_1=E_2$ en esta familia y entonces $d_1$ sería siempre nulo, por lo que tomo el producto por otra balanced con $d_1\not=0$.} So, we will proceed as follows.

Let $X^3$ be any compact balanced nilmanifold with Fr\"olicher spectral sequence not degenerating at the first page. We can consider, for instance, the Iwasawa manifold or any of those studied in~\cite{COUV}. Let $Z^7_{\alpha}=X^3\times Y^4_{\alpha}$, with $\alpha\geq 0$ and $\alpha\in\Q$. A similar argument as above 
allows us to conclude that $Z^{7}_{\alpha}$ and $Z^{7}_{\alpha'}$ have different complex homotopy type whenever $\alpha\neq \alpha'$, they are balanced and their 
Fr\"olicher spectral sequence $\{E_r(Z^{7}_{\alpha})\}_{r\geq 1}$ is not degenerating at the $1$-st page (neither at the $2$-nd page). 
%%%
\begin{comment}
By Theorem~\ref{theorem-balanced} we have that any complex nilmanifold $X^4_{\gamma}$ with underlying Lie algebra $\frg_1^\gamma$, $\gamma\in\{0,1\}$, admits balanced metrics. Moreover,  
$E_1^{0,2}(X^4_{\gamma})\not= E_2^{0,2}(X^4_{\gamma})$ by Table~1. 
Consider the complex eight dimensional manifold $Z^{8}_{\alpha,\gamma}=X^4_{\gamma}\times Y^4_{\alpha}$. Now, a similar argument as above 
allows us to conclude that $Z^{8}_{\alpha,\gamma}$ and $Z^{8}_{\alpha',\gamma'}$ have different complex homotopy type whenever $\alpha\neq \alpha'$ or $\gamma\neq \gamma'$, they are balanced and their 
Fr\"olicher spectral sequence $\{E_r(Z^{8}_{\alpha,\gamma})\}_{r\geq 1}$ is not degenerating at the $1$-st page (neither at the $2$-nd page). Indeed,
$$
%\begin{equation}\label{kunneth-3}
\begin{array}{rl}
e_1^{0,2}(Z^{8}_{\alpha,\gamma})
=&e_1^{0,2}(X) \cdot e_1^{0,0}(Y^4_{\alpha})+\sum_{l=1}^{2} e_1^{0,2-l}(X) \cdot e_1^{0,l}(Y)\\
>&e_{2}^{0,2}(X) \cdot e_{2}^{0,0}(Y)+ \sum_{l=1}^{4} e_{2}^{0,2-l}(X) \cdot e_{2}^{0,l}(Y)\ = e_{2}^{0,2}(Z^{8}_{\alpha,\gamma}). 
\end{array}
%\end{equation}
$$
\end{comment}
%%%
\end{proof} 

%%%It is worthy to remark that the compact balanced manifolds $Z$ given in Theorem~\ref{th-inf-balanced} for $n\geq 3$ also have $d_2\not=0$, so they satisfy
%%%$$\cdots\  E_2(Z)\not\cong E_3(Z) \  \cdots \  E_n(Z)\not\cong E_{n+1}(Z)\ \cdots.$$
%%%In other words, $E_3(Z)$ is a complex invariant of the manifold $Z$ which is different from the Dolbeault cohomology $E_1(Z)$ and the De Rham cohomology $E_{\infty}(Z)$.

\smallskip

It is worthy to remark that, in the result concerning the second page,  
%for the case $n=2$ 
the dimension of $Z$ is optimal in the class of complex nilmanifolds;  
%(with invariant complex structure). 
%%%Recall that any balanced nilmanifold of complex dimension 3 degenerates at the second page; in addition, 
indeed, 
Bazzoni and Mu\~noz prove in~\cite{BM2012} that the number of real homotopy types of $6$-dimensional nilmanifolds is finite.

However, for other pages %other values of $n$, 
the dimension of the balanced manifolds $Z$ in the infinite families seems to be far from being optimal, even in the class of complex nilmanifolds. 
For instance, for the third page %in the case $n=3$ 
we can consider the (3-step) nilmanifold of real dimension 12 endowed with a nilpotent complex structure given in \cite{CFG-illinois}. 
Let us denote by $X^6$ this complex nilmanifold, whose complex structure equations are 
$$
d\omega^1=d\omega^2=d\omega^3=0, \quad
d\omega^4=\omega^{12}+\omega^{1\bar{2}}, \quad
d\omega^5=-\omega^{2\bar{1}}, \quad
d\omega^6=\omega^{14}+\omega^{1\bar{3}}.
$$
It is proved in \cite{CFG-illinois} that $E_3(X^6)\not\cong E_4(X^6)$. 
For  any $\rho=(\rho_1,\ldots,\rho_{6})\in (\R^+)^{6}$, one can check that the ``diagonal'' invariant Hermitian metrics on $X^6$ defined by 
$F_{\rho}=\frac{i}{2} \sum_{k=1}^6 \rho_k \,\omega^{k\bar{k}}$ are balanced. %%%\footnote{(1) Comprobar esta afirmaci\'on; (2) Mirar si por casualidad este ejemplo de \cite{CFG-illinois} admite SKT. Adela: Comprobado (1); respecto a (2), no hay SKT} 
Now, a similar argument as the one given in the proof of Theorem~\ref{th-inf-balanced} 
shows that the infinite family of balanced manifolds $Z^{10}_{\alpha}=X^6\times Y^{4}_{\alpha}$ satisfies the required properties for $n=3$ in complex dimension 10.

\smallskip

Finally, it is still unclear whether there are obstructions in the Fr\"olicher spectral sequence under the existence of a balanced metric on a compact complex manifold. We recall that in the class of balanced nilmanifolds of complex dimension 3 the Fr\"olicher spectral sequence always degenerates at the second page~\cite{COUV}. We believe that the following general question could have a positive answer:

\smallskip

\noindent {\bf Question.} Let $X$ be a compact balanced manifold with $\dim_{\C} X=n$. 
Does the Fr\"olicher spectral sequence of $X$ degenerate at a $k$-th page for some $k\leq n-1$?

%%%%%%%%%%%%%%%%%%%%%%%%%%%%%%%%%%%
\section{Fr\"olicher spectral sequence of SKT and generalized Gauduchon manifolds 
%A counterexample to a conjecture by Popovici
}\label{counterexample-section}
%%%%%%%%%%%%%%%%%%%%%%%%%%%%%%%%%%%

\noindent In this section we construct compact complex manifolds endowed with SKT 
and, more generally, generalized Gauduchon metrics having non-degenerate Fr\"olicher spectral sequence. 

Let us first recall the definition of generalized Gauduchon metrics, introduced and studied by Fu, Wang and Wu in \cite{FWW}.

\begin{definition} \cite{FWW}
Let $X$ be a compact complex manifold with $\dim_{\C}X=n$, and let %$k$ be an integer such that
$1 \leq k \leq n -1$ be an integer.
A Hermitian metric $F$ on $X$ is called {\em $k$-th Gauduchon} if it satisfies the condition
\begin{equation*}% \label{defk-Gaud}
\partial \db F^k  \wedge F^{n -k-1} =0.
\end{equation*}
\end{definition}

From the definition, it is clear that the value $k = n-1$ corresponds to the classical \emph{standard} (also known as \emph{Gauduchon}) metrics \cite{Gau}. Moreover, observe that any SKT metric ($\partial \db F=0$) is in particular $1$-st Gauduchon.

We recall that by \cite[Theorem 2.5]{FU}, %when $n\geq 3$\footnote{Me queda descolgado, ?`se refiere a dimensi\'on? Quiz\'a haya que indicarlo}, 
any conformally balanced $1$-st Gauduchon metric $F$ on a compact
complex manifold $X$ with $\dim_{\C}X=n\geq 3$ %%%with $\dim_{\C}X=n\geq 3$ 
whose Bismut connection has (restricted)
holonomy contained in SU($n$) is necessarily K\"ahler.
%%%and therefore it is a Calabi-Yau structure. 
In \cite{IP} Ivanov and Papadopoulos
extend this result to any generalized $k$-th Gauduchon metric, for any $k \neq n-1$.

Very little is known about the relation between the existence of $k$-th Gauduchon metrics, for some $1 \leq k \leq n -2$, on a compact complex manifold $X$ and the degeneration of its FSS. A first consequence of the results in the preceding sections is the following one:
%%%Section~\ref{FSS-section} together with %Remark~\ref{iso-complejo} 
%%%the proof of Theorem~\ref{cor-infinite-bal-FSS} and \cite[Theorem 5.2]{LUV2} is the following result in complex dimension 4.
%%%\footnote{(otra opci\'on) is the existence of infinitely many complex (hence, real or rational) homotopy types of  compact $1$-st Gauduchon manifolds $Z$ with $\dim_{\C}Z=4$ and Fr\"olicher spectral sequence not degenerating at the second page.}

\begin{proposition}\label{cor-infinite-1-G-FSS}
There are infinitely many complex %%%(hence, real or rational) 
homotopy types of  compact $1$-st Gauduchon manifolds $Z$ with $\dim_{\C}Z=4$ and Fr\"olicher spectral sequence not degenerating at the second page.
\end{proposition}

\begin{proof}
It is enough to consider the nilmanifolds $Y^4_{\alpha}$ with underlying Lie algebra isomorphic to $\frg_{11}^{q,1}$ with $q\geq 0$ and $q\in \Q$, given in the proof of Theorem~\ref{cor-infinite-bal-FSS}. In fact, these nilmanifolds admit generalized Gauduchon metrics by \cite[Theorem 5.2]{LUV2}.
\end{proof}

Next, we extend this result to obtain non-degeneration at an arbitrarily large page. 
Let $X$ be a complex nilmanifold with $\dim_{\C}X =n\geq 2$. 
For any invariant Hermitian metric $F$ on $X$,
the real $(n,n)$-form
$\frac{i}{2}\,\partial\bar{\partial}F\wedge F^{n-2}$ is proportional to the volume form $F^n$.
Therefore,
\begin{equation}\label{formula-constante}
\frac{i}{2}\,\partial\bar{\partial}F\wedge F^{n-2}= c_1(F)\, F^n,
\end{equation}
for some constant $c_1(F) \in \mathbb{R}$. By \cite[Proposition 11]{FWW} the sign of the constant $c_1(F)$ is an invariant of the conformal class of $F$. Observe that $F$ is $1$-st Gauduchon if and only if~$c_1(F)=0$. Note that when $n=2$ the constant $c_1(F)=0$ because any invariant metric is standard. 

\begin{proposition}\label{product}\cite[Proposition 2.3 and Corollary 2.4]{LUV-dga}
Let $X$ and $X'$ be complex nilmanifolds %%%of complex dimension $m$ and $m'$
endowed with invariant Hermitian metrics $F$ and $F'$, respectively.
\begin{enumerate}
\item[(i)] For any real positive number $\lambda>0$, we have
%\begin{equation}\label{formula1}
$c_1(\lambda\, F)=\lambda^{-1} c_1(F)$.
%\end{equation}
%
\item[(ii)] The product Hermitian metric $F + F'$ on the nilmanifold $X \times X'$ satisfies
%\begin{equation}\label{formula2}
$$c_1(F + F')={n\,(n-1) \over (n+n')(n+n'-1)}\, c_1(F)+ {n'(n'-1) \over (n+n')(n+n'-1)}\, c_1(F'),$$
%\end{equation}
where $n=\dim_\mathbb{C} X$ and $n'=\dim_\mathbb{C} X'$.
In particular, if $c_1(F)>0$ and $c_1(F')<0$, 
then $X \times X'$ has a $1$-st Gauduchon metric.
\end{enumerate}
\end{proposition}

It is worth to recall that for complex nilmanifolds of complex dimension 3, the  existence of an invariant Hermitian metric $F$ with $c_1(F)<0$ implies the existence of a $1$-st Gauduchon metric (possibly non-invariant and non-SKT) on the nilmanifold (see \cite[Theorem 3.6]{FU} for details). In \cite[Proposition 2.9]{LUV-dga} a classification of complex structures admitting Hermitian metrics with $c_1<0$ is given. 
Using such classification together with \cite[Theorem 4.1]{COUV}, one has that the FSS of a complex nilmanifold $X$ of complex dimension 3 degenerates at the second page whenever an invariant Hermitian metric with $c_1\leq 0$ exists on $X$. 

\smallskip

We use the results obtained in Section~\ref{balanced-section} together with the proposition above to extend Proposition~\ref{cor-infinite-1-G-FSS} to arbitrarily high pages.

\begin{theorem}\label{th-inf-k-G}
For every integer $n\geq 3$, 
there exist infinitely many compact $1$-st Gauduchon  
manifolds $W$ with $\dim_{\C}W=4n+5$ and different complex (hence, real or rational) homotopy types whose Fr\"olicher spectral sequence $\{E_r(W)\}$ 
is not degenerating at the $n$-th page.

Similar results hold for the first and second pages when complex dimensions $7$ and $4$ are respectively considered.

Furthermore, all the previous compact complex manifolds are $k$-th Gauduchon for every~$k$.
\end{theorem}

\begin{proof} 
First, we prove the theorem for every $n\geq 3$. For each $\alpha\geq 0$ and $\alpha\in\Q$, let $Z^{4n+2}_{\alpha}$ be a compact balanced manifold in the family given in 
Theorem~\ref{th-inf-balanced}. 
Recall that $Z_{\alpha}$ is a complex nilmanifold and it has a (invariant) balanced metric $F$. 
By \cite[Lemma 3.7]{IP}, one has that the constant $c_1(F) > 0$. 

Let us now consider any complex nilmanifold $X'$ of complex dimension 3 endowed with an invariant Hermitian metric $F'$ with $c_1(F')<0$. Note that these nilmanifolds are classified in \cite[Proposition 2.9]{LUV-dga}, and three is the lowest possible dimension where this can occur.  

Let $W_{\alpha}=Z_{\alpha}\times X'$. It follows from Proposition~\ref{product} that $W_{\alpha}$ has a $1$-st Gauduchon metric. 
Furthermore, a similar argument as in the proof of Theorem~\ref{th-inf-balanced} implies that for any non-negative $\alpha,\alpha'\in\Q$ with $\alpha\not=\alpha'$, the (nil)manifolds $W_{\alpha}=Z_{\alpha}\times X'$ and $W_{\alpha'}=Z_{\alpha'}\times X'$ have different complex homotopy types because their minimal models are not isomorphic (over $\C$). Applying the K\"unneth formula to $W_{\alpha}=Z_{\alpha}\times X'$
we conclude that $d_n\not=0$ for all the manifolds $W_{\alpha}$.
This gives the first part of the statement.

\smallskip

The result for the second page comes directly from Proposition~\ref{cor-infinite-1-G-FSS}, so we focus on the non-degeneration at the first page. 
Let $X'$ be a $3$-dimensional complex nilmanifold with Fr\"olicher spectral sequence not degenerating at the first page and having an invariant Hermitian metric $F$ with $c_1(F)<0$. For example, one can consider that $X'$ is defined by the following complex structure equations
$$
d\omega^1=d\omega^2=0,\quad d\omega^3=\omega^{1\bar{1}}+(1+i)\omega^{2\bar{2}},
$$
and then the Hermitian metric $F=\frac{i}{2}(\omega^{1\bar{1}}+\omega^{2\bar{2}}+\omega^{3\bar{3}})$ satisfies $\frac{i}{2}\,\partial\bar{\partial}F\wedge F= \frac{-1}{12}\, F^3$, for instance.
Let $Z^7_{\alpha}=X'\times Y^4_{\alpha}$, with $\alpha\geq 0$ and $\alpha\in\Q$, where $Y^4_{\alpha}$ is any of the complex balanced manifolds in the  infinite family given in Theorem~\ref{cor-infinite-bal-FSS}.  A similar argument as above 
allows us to conclude the result for the first page in seven complex dimensions.

\smallskip

For the final statement, we just apply \cite[Proposition 2.2]{LU-cr}, where it is proved that if an invariant Hermitian
metric $F$ on a complex nilmanifold $W$ of complex dimension $n\geq 4$
is $k_0$-th Gauduchon for some $k_0$ with $1 \leq k_0 \leq n-2$, 
then $F$ is $k$-th Gauduchon \emph{for every} $k$ with $1 \leq k \leq n-1$. 
Alternatively, one can also apply \cite[Proposition 3.1]{OOS}.
\end{proof} 

Observe that the generalized Gauduchon metrics on the manifolds given in the previous result are
not SKT. Indeed, 
Arroyo and Nicolini prove in \cite[Theorem~1.2]{AN} that any complex nilmanifold admitting an invariant SKT metric is either a torus or 2-step nilpotent. Since any 
nilmanifold with SnN complex structure has nilpotency step $s\geq 3$ (see \cite[Corollary 3.6]{LUV1}), 
in particular the manifolds constructed in Theorems~\ref{th-inf-balanced} and~\ref{th-inf-k-G} do not admit any SKT metric. Note also that the
Bigalke-Rollenske nilmanifolds $X$ are 2-step but they do not admit any SKT metric by \cite[Proposition~3.5]{ST21}. The latter also follows from the fact that the equations \eqref{ecusBR-2} for $X$ imply that the underlying Lie algebra $\frg$ satisfies that $[\frg,\frg]+J[\frg,\frg]$ is abelian and,  
under this condition, Fino and Vezzoni proved 
(see \cite[Theorem~1.1]{FV2} and \cite[Theorem~A]{FV3}) 
that if $X$ carries an SKT metric and a balanced metric, then $X$ is necessarily a complex torus.

In \cite[Conjecture 1.3]{Pop1}, it is conjectured that \emph{any compact complex manifold $X$ admitting an SKT metric has
Fr\"olicher spectral sequence degenerating at the second page}. 
Some evidence for this conjecture is provided in \cite{Pop1,Pop2}, where Popovici obtains sufficient metric conditions for the $E_2$ degeneration
of the Fr\"olicher spectral sequence. The main idea is that the existence of a Hermitian metric on $X$ with \emph{small torsion} (in the sense that we recall below) implies that $E_2(X)=E_{\infty}(X)$.

%In \cite{Pop1} Popovici studied the degeneration of the FSS at $E_2$ under certain conditions on the compact complex manifold. 
%%%Given a Hermitian metric on $X$, an associated Laplace-type operator is introduced whose kernel is isomorphic to $E_2^{p,q}$ in every bidegree $(p, q)$. 
Let $F$ be any Hermitian metric on a compact complex manifold $X$. Recall that the torsion operator
of order zero and bidegree $(1,0)$ associated with $F$  (see \cite[VII {\S}1]{Dem97}) is given 
by $\tau :=[\Lambda, \partial F\wedge\cdot\,]$, where $\Lambda$ is the formal adjoint of the Lefschetz operator
$L := F \wedge \cdot$ with respect to the $L^2$-inner product $\langle\!\langle\cdot,\cdot\rangle\!\rangle$ induced by $F$ on differential forms.
Define 
$C^{p,q}_F:= \sup_{u\in \Omega^{p,q}(X),\, \|u\|=1} \langle\!\langle
([\tau, \tau^\star] + [\partial F \wedge\cdot, (\partial F \wedge\cdot)^\star])\,u, u \rangle\!\rangle$.
Now, let $\Delta',\Delta''\colon \Omega^{p,q}(X) \longrightarrow \Omega^{p,q}(X)$ be the
usual Laplace-Beltrami operators, i.e. $\Delta'=\partial\partial^\star+\partial^\star\partial$ and $\Delta''=\db\db^\star+\db^\star\db$. The non-negative self-adjoint differential operator $\Delta'+\Delta''$
%%%$\Delta'+\Delta''\colon \Omega^{p,q}(X) \longrightarrow \Omega^{p,q}(X)$ 
is elliptic and, since $X$ is compact, it has a
discrete spectrum contained in $[0,+\infty)$ with $+\infty$ as its only accumulation point. Denote by 
$\rho^{p,q}_F:= \min({\rm Spec}(\Delta'+\Delta'')^{p,q} \cap (0,+\infty))$ its smallest positive eigenvalue. 
Thus, $\rho^{p,q}_F$ is the size of the spectral gap of $\Delta'+\Delta''$ acting on $(p, q)$-forms.

Popovici proves in \cite[Theorem 5.4]{Pop1} that if a compact complex $n$-dimensional manifold $X$ carries an SKT metric $F$ whose torsion satisfies the condition
$C^{p,q}_F \leq \frac13 \rho^{p,q}_F$ 
for all $p, q \in \{0,\ldots, n\}$, then the Fr\"olicher spectral sequence of $X$ degenerates at $E_2(X)$.
Furthermore, in \cite[Theorem~1.2]{Pop2} it is proved the degeneration at $E_2$ occurs whenever the manifold admits a Hermitian metric $F$ satisfying $\ker \Delta'' \subset \ker\, [\tau, \tau^\star]$, that is, the torsion operator and its adjoint vanish on $\Delta''$-harmonic forms.

\smallskip

We next provide a counterexample to the previous conjecture, so the torsion of SKT metrics may not be small in general.
The construction is based on the complex geometry of compact Lie groups.\footnote{We were informed by Jonas Stelzig 
of the existence of SKT nilmanifolds with arbitrarily non-degenerate FSS.}

\smallskip

%%%We next provide a counterexample to the previous conjecture, based on the complex geometry of compact Lie groups. This shows that the torsion of SKT metrics may not be small in general. Let us remark that the compact SKT manifold with $E_2\neq E_{\infty}$ that gives the counterexample has complex dimension 18, so the question remains open in lower dimensions. We also ignore whether the conjecture may hold true in the class of complex nilmanifolds.

\smallskip

Let $G$  be a connected Lie group with Lie algebra $\frg$, and denote by $\gc$ the complexification of $\frg$. Giving a left-invariant almost complex structure $J$ on $G$ is equivalent to the choice of a subspace $\frs \subset \gc$ such that
$$
\frs \cap \frg =\{0\}, \quad \gc=\frs\oplus\overline{\frs}.
$$
Hence $\frs$ is the subspace of $(1,0)$-elements. 
Now, $J$ is integrable if and only if $\frs$ is a subalgebra of $\gc$, i.e. $[\frs,\frs]\subset \frs$.

For an even-dimensional compact Lie group $G$, Samelson provided in \cite{Samelson} a construction of left-invariant complex structures on $G$ that we briefly recall here. Let $T$ be a maximal torus in $G$ with Lie algebra $\frt$, and suppose that $\alpha_1,\ldots,\alpha_r\in \frt^*$ is a set of positive roots. Here $2r$ is the rank of $G$. Then we have the ${\rm ad}(T)$-invariant decomposition
$$
\gc=\tc\oplus \sum_{j=1}^r\frs_{\alpha_j}\oplus \sum_{j=1}^r\frs_{-\alpha_j},
$$
where $\tc$ is the complexification of $\frt$, and $\frs_{\alpha}$ is given by
$$
\frs_{\alpha}:=\{Z\in \gc \mid [x,Z]=2\pi i \alpha(x)Z \quad\forall x\in\frt\}.
$$
So, if we choose a subspace $\fra\subset \tc$ such that $\fra \cap \frt =\{0\}$ and $\fra\oplus\overline{\fra}=\tc$, then we get a left-invariant complex structure $J$ on $G$ defined by 
$$
\frs=\fra\oplus \sum_{j=1}^r\frs_{\alpha_j}.
$$

Now, let $g$ be a bi-invariant metric on $G$, and denote by $\langle \ ,\ \rangle$ its $\C$-linear extension. The compatibility of the complex structure $J$ with $g$ is equivalent to $\frs$ being isotropic, i.e. $\langle \frs ,\frs \rangle=0$. 
Alexandrov and Ivanov prove in \cite{AI} that any compact Lie group equipped with a  bi-invariant metric $g$ and a left-invariant complex
structure $J$ compatible with $g$ is Bismut flat and its fundamental form $F$ is $dd^c$-harmonic. The latter condition means that $F$ is SKT (since $dd^cF=0$) and standard (since $(dd^c)^*F=0$ is equivalent to $\partial\db F^{n-1}=0$, where $2n=\dim_{\R}G$). This result is used in \cite{AI} to compute the Hodge numbers $h^{0,q}$ of compact Lie groups with such a Hermitian structure. 

Bismut flat manifolds play a relevant role in relation to the pluriclosed flow. In \cite{GJS}, Garcia-Fernandez, Jordan and Streets prove global existence and convergence of the pluriclosed flow and the generalized K\"ahler Ricci flow on compact Bismut flat manifolds. Note that by \cite{WYZ} it turns out that, up to taking universal covers, all such manifolds are given by the above Samelson spaces. 
%previously described.

In what follows, we consider the compact semisimple Lie group SO(9) equipped with the bi-invariant metric $g$ given by minus the Killing form and with a left-invariant complex structure $J$ compatible with $g$ found by Pittie in \cite{Pittie-indian,Pittie-bull}, as we next recall. The Lie group $G={\rm SO(9)}$ has rank four, so we choose a maximal 4-torus $T$ and a basis for its Lie algebra $\frt$ given by
$$
\frt=\langle e_1,e_2,e_3,e_4 \rangle
$$
so that $\{e_k\}_{k=1}^4$ is orthonormal for the metric $g$. 
Consider the subspace $\fra\subset \tc$ defined by
$$
\fra=\langle e_1+e_2+i\sqrt{2}\,e_3\, ,\, e_1-e_2+i\sqrt{2}\,e_4 \rangle.
$$
It is clear that $\fra \cap \frt =\{0\}$ and $\fra\oplus\overline{\fra}=\tc$, then one has a left-invariant complex structure $J$ on $G$. Moreover, $\fra$ is isotropic, so also is $\frs$, and thus $J$ is compatible with the metric $g$.

\begin{proposition}\label{counterex}
The $18$-dimensional compact complex manifold $X=({\rm SO}(9),J)$ satisfies  $E_2(X)\neq E_{\infty}(X)$ and it has an SKT metric, which is in addition $k$-th Gauduchon for every $1 \leq k \leq 17$.
\end{proposition}

\begin{proof}
Let us consider the space
$$
\Lambda V=\Lambda(w_{2,1},w_{4,3},w_{6,5},w_{8,7})\otimes\Lambda(v_1,v_2)\otimes\Lambda(u_1,u_2),
$$
where the generators $w$'s have the bidegree indicated by the subindices, $v_1,v_2$ have bidegree $(1,1)$ and $u_1,u_2$ have bidegree $(0,1)$. Recall that, since the total degree of $v_j$ is even, the space 
$\Lambda(v_1,v_2)$ is a polynomial algebra.

Let us consider $\db$ defined on generators by
$$
\db u_1= \db u_2= \db v_1= \db v_2= \db w_{2,1} =0,\quad 
\db w_{4,3} =f^2,\quad \db w_{6,5} =fg,\quad \db w_{8,7} =f^2,
$$
where $f=v_1^2+v_2^2$ and $g=v_1^2v_2^2$, and let $\partial$ be given by
$$
\partial u_1= v_1,\quad  \partial u_2= v_2, \quad \partial v_1= \partial v_2= \partial w_{2,1} = \partial w_{4,3} =\partial w_{6,5} =\partial w_{8,7} =0.
$$
Pittie proved (see also \cite{Tanre}) that $(\Lambda V,\db)$ is the Dolbeault minimal model of $X$, and used it to show that the FSS of $X$ does not degenerate at $E_2$. 
For the seek of completeness, we here show that  the map
$$
d_2\colon E_2^{6,8}\longrightarrow E_2^{8,7}
$$ 
is not identically zero. Let $\xi:=u_1v_1+u_2v_2$ and $\eta:=u_1v_1v_2^2$. 
We have $\partial\xi=f$, $\partial\eta=g$. The element $f\xi\eta$ has bidegree $(6,8)$, it is clearly $\db$-closed and $\partial (f\xi\eta) = f^2\eta-fg\xi$. 
Since
$$
\partial (f\xi\eta) + \db(w_{6,5}\,\xi-w_{4,3}\,\eta)=0,
$$
by the description \eqref{E=X/Y}-\eqref{dr} we have that $f\xi\eta\in E_2^{6,8}$ and 
$$
d_2(f\xi\eta)=\partial(w_{6,5}\,\xi-w_{4,3}\,\eta)=w_{4,3}\,g-w_{6,5}\,f \in E_2^{8,7}.
$$
Using \eqref{Ypq}, it is a direct calculation to check that $w_{4,3}\,g-w_{6,5}\,f \not\in {\mathcal Y}^{8,7}_2$, so $d_2$ is non-zero and $E_2^{6,8}\neq E_3^{6,8}$.

Alternatively, Tanr\'e showed that $w_{4,3}\,g-w_{6,5}\,f$ is a non-zero Dolbeault-Massey product and, by \cite[Th\'eor\`eme 9]{Tanre}, the degeneration of the FSS at $E_2$ is equivalent to the Dolbeault formality, for every even-dimensional compact connected Lie group $G$ such that $T \rightarrow G \rightarrow G/T$ is a principal holomorphic fiber bundle, where $T$ is a maximal torus in $G$. 

Finally,
by \cite{AI} the fundamental form $F$ is $dd^c$-harmonic, so it is SKT and Gauduchon. The last assertion in the proposition follows from \cite[Corollary 2.3]{LU-cr} because $F$ is left-invariant.
%Let X be a homogeneous compact complex manifold of complex dimension $n \geq 3$ and let F be an invariant Hermitian metric on X. If F is SKT or astheno-Kahler, then F is $k$-th Gauduchon for any $1 \leq k \leq n -1$.
%%%we get that $F$ is $k$-th Gauduchon for every $1 \leq k \leq n -1$.
\end{proof}

Note that $E_r(X)= E_{\infty}(X)$ for every $r\geq 3$, so the FSS of $X$ degenerates at the third page (see \cite{Pittie-indian,Pittie-bull} for details).

%%%Pittie observed that any translate of (the complex 2-plane) $\fra_0\subset \frh$ by an element of the Weyl group gives an isomorphic Dolbeault ring and an isomorphic $E_2$ term in the FSS sequence. Thus there are a finite set of counter-examples to the Frolicher conjecture. (PERO esas traslaciones son compatibles con la métrica bi-invariante $g$ ???

\vskip.2cm

\appendix
%%%%%%%%%%%%%%%%%%%%%%%%%%%%%%%%%%%%%%%%%%
%%%%%%%%%%%%%%%%%%%%%%%%%%%%%%%%%%%%%%%%%%%

\section{The nilmanifolds $Y^4_{\alpha}$ have pairwise non-isomorphic $\C$-minimal models}%%%\label{A}

%\vskip.2cm

\noindent In this section we complete the proof of Theorem~\ref{cor-infinite-bal-FSS} by showing that the Lie algebras $\mathfrak \frg_{11}^{\alpha, 1}$, $\alpha\in[0,+\infty)$, underlying $Y^4_{\alpha}$ are pairwise non-isomorphic over $\C$. 
%For simplicity, we denote $\mathfrak \frg_{11}^{\alpha, 1}$ by $\mathfrak m_{\alpha}$. 
Recall that the structure equations of $\mathfrak \frg_{11}^{\alpha, 1}$ are
%$$\mathfrak \frg_{11}^{\alpha, 1} = \big(0^3,\, 13,\, 23,\, 14+25-35,\, \alpha \!\cdot\! 12+15+24+34,\, 16+27-45-(2\!\cdot\! 25 + 35)\big),$$
%where $\alpha\in[0,+\infty)$.
\begin{equation}\label{ec-reales-FII-caseiii}
%\left\{
\begin{split}
dv^1 &= dv^2 =dv^3= 0,\quad
dv^4 = v^{13},\quad
dv^5 = v^{23},\quad
dv^6 = v^{14}+v^{25}-v^{35},\\[4pt]
dv^7 &= \alpha\,v^{12}+v^{15}+v^{24}+v^{34},\quad
dv^8 = v^{16}-2\,v^{25}+v^{27}-v^{35}-v^{45}.
\end{split}
%\right.
\end{equation}

We will prove that if $f\colon \mathfrak \frg_{11}^{\alpha, 1} \longrightarrow \mathfrak \frg_{11}^{\alpha', 1}$ is an isomorphism of Lie algebras, then $\alpha=\alpha'$. 
Note that the dual map $f^* \colon (\mathfrak \frg_{11}^{\alpha', 1})^* \longrightarrow (\mathfrak \frg_{11}^{\alpha, 1})^*$ extends to a map
$F \colon \bigwedge^*(\mathfrak \frg_{11}^{\alpha', 1})^* \longrightarrow \bigwedge^*(\mathfrak \frg_{11}^{\alpha, 1})^*$ that commutes with the differentials, i.e. $F\circ d=d\circ F$.

Let $\{v^k\}_{k=1}^8$ (resp. $\{v'^{\,k}\}_{k=1}^8$) be a basis for~$(\mathfrak \frg_{11}^{\alpha, 1})^*$
(resp. $(\mathfrak \frg_{11}^{\alpha', 1})^*$) satisfying the equations~\eqref{ec-reales-FII-caseiii}
for $\alpha$ (resp. $\alpha'$).
In terms of these bases, $F$ is determined by
$$
%%%\begin{equation}\label{cambio-base}
F(v'^{\,k}) =\sum_{j=1}^8 \lambda_j^k\, v^j, \quad k=1,\ldots,8,
%%%\end{equation}
$$
where the matrix $\Lambda=(\lambda^k_j)_{k,j=1,\ldots,8}$ belongs to ${\rm GL}(8,\C)$, 
and the condition $F\circ d=d\circ F$ reads as
\begin{equation}\label{condiciones}
F(dv^k) - d(F(v'^{\,k}))=0, \ \text{ for each }1\leq k\leq 8.
\end{equation}

In \cite{LUV2} the isomorphism problem for this family of algebras was studied in the case of the field of real numbers. 
We notice that \cite[Lemma 3.4]{LUV2} is still valid over $\C$, as it is the first part of the proof of \cite[Lemma 3.5]{LUV2}. Thus, one arrives at a complex matrix $\Lambda=(\lambda^k_j)_{k,j}$ of the form
\begin{equation}\label{matriz-reducida}
\Lambda=
\begin{pmatrix}
\lambda^1_1 & \lambda^1_2 & 0 & 0 & 0 & 0 & 0 & 0 \\
\lambda^2_1 & \lambda^2_2 & 0 & 0 & 0 & 0 & 0 & 0 \\
0 & 0 & \lambda^3_3 & 0 & 0 & 0 & 0 & 0 \\
\lambda^4_1 & \lambda^4_2 & \lambda^4_3 & \lambda^4_4 & \lambda^4_5 & 0 & 0 & 0 \\
\lambda^5_1 & \lambda^5_2 & \lambda^5_3 & \lambda^5_4 & \lambda^5_5 & 0 & 0 & 0 \\
\lambda^6_1 & \lambda^6_2 & \lambda^6_3 & \lambda^6_4 & \lambda^6_5 & \lambda^6_6 & \lambda^6_7 & 0 \\
\lambda^7_1 & \lambda^7_2 & \lambda^7_3 & \lambda^7_4 & \lambda^7_5 & \lambda^7_6 & \lambda^7_7 & 0 \\
\lambda^8_1 & \lambda^8_2 & \lambda^8_3 & \lambda^8_4 & \lambda^8_5 & \lambda^8_6 & \lambda^8_7 & \lambda^8_8 
\end{pmatrix},
\end{equation}
where $\lambda^1_1 \lambda^2_2-\lambda^1_2\lambda^2_1\not=0$, $\lambda^3_3,\lambda^8_8\not=0$, together with 
$$
\begin{array}{lllll}
&\lambda^4_4 = \lambda^{1}_1\,\lambda^3_3, \quad &\lambda^4_5 = \lambda^{1}_2\,\lambda^3_3, \quad
&\lambda^5_4 = \lambda^{2}_1\,\lambda^3_3, \quad &\lambda^5_5 = \lambda^{2}_2\,\lambda^3_3, \\
&\lambda^6_6=\lambda^2_2\,(\lambda^3_3)^2, \quad &\lambda^6_7=-\lambda^2_1\,(\lambda^3_3)^2, \quad
&\lambda^7_6=-\lambda^1_2\,(\lambda^3_3)^2, \quad &\lambda^7_7=\lambda^1_1\,(\lambda^3_3)^2,
\end{array} 
$$
and
\begin{equation}\label{ecu-principal}
(\lambda^1_2)^2 + (\lambda^2_2)^2 - \lambda^2_2\,\lambda^3_3 = 0, \qquad
\lambda^1_2\,(\lambda^3_3 + 2\,\lambda^2_2) = 0.
\end{equation}
We distinguish two cases depending on the vanishing of the complex coefficient $\lambda^1_2$.

Let us first suppose that $\lambda^1_2=0$. 
Note that in this case the second part of the proof of \cite[Lemma 3.5]{LUV2} and \cite[Lemma 3.6]{LUV2} are still valid for $\C$. Hence, for $\alpha\not=\alpha'$, any $F$ with $\lambda^1_2=0$  
is not an isomorphism over $\C$. 

Consequently, we will now assume that $\lambda^1_2\neq 0$. Then, the second expression in \eqref{ecu-principal} gives
$\lambda^3_3=-2\,\lambda^2_2$, and replacing it in the first one, we get
\begin{equation}\label{ecu-p1}
(\lambda^1_2)^2 + 3\, (\lambda^2_2)^2 = 0.
\end{equation}
On the other hand, the coefficients of $v^{14}$ and $v^{15}$ in the condition \eqref{condiciones} for $k=7$ are, respectively, $2\,\lambda^1_1\,\lambda^2_1 +\lambda^1_2\,\lambda^3_3$ and $\lambda^1_2\,\lambda^2_1 +\lambda^1_1(\lambda^2_2-\lambda^3_3)$, so using again  $\lambda^3_3=-2\,\lambda^2_2$ we get
\begin{equation}\label{ecu-p2}
\lambda^1_1\,\lambda^2_1 -\lambda^1_2\,\lambda^2_2 =0,\quad \quad
\lambda^1_2\,\lambda^2_1 +3\,\lambda^1_1\,\lambda^2_2 =0.
\end{equation}

The equations \eqref{ecu-p1}-\eqref{ecu-p2} can be solved explicitly in terms of $\lambda^1_2$. Notice that $(\lambda^1_1)^2=(\lambda^2_2)^2$. There are four solutions:
\begin{equation}\label{sol-p1-p2}
\lambda^1_1= \vartheta\xi \frac{i}{\sqrt{3}} \lambda^1_2, \quad
\lambda^2_1= \xi\,\lambda^1_2, \quad 
\lambda^2_2= \vartheta \frac{i}{\sqrt{3}} \lambda^1_2,\quad
\lambda^3_3= -2\,\vartheta \frac{i}{\sqrt{3}} \lambda^1_2, \quad\quad \vartheta,\xi\in\{\pm 1\}.
\end{equation}

We will use these solutions below, after taking into account first other equations coming from the conditions \eqref{condiciones} for $k=6,7,8$ (note that \eqref{condiciones} is already fullfilled for $1\leq k\leq 5$).

The coefficients of $v^{13}$ and $v^{23}$ in the condition \eqref{condiciones} for $k=6,7$ give the following equations:
\begin{equation*}
\begin{array}{lll}
&\lambda^6_4=\lambda^1_1 \lambda^4_3 + \lambda^3_3 \lambda^5_1 +\lambda^2_1 \lambda^5_3 , \quad\quad
&\lambda^6_5=\lambda^1_2 \lambda^4_3 + \lambda^3_3 \lambda^5_2 +\lambda^2_2 \lambda^5_3 , \\[2pt] 
&\lambda^7_4= -\lambda^3_3 \lambda^4_1 + \lambda^2_1 \lambda^4_3 +\lambda^1_1 \lambda^5_3 , \quad\quad
&\lambda^7_5= -\lambda^3_3 \lambda^4_2 + \lambda^2_2 \lambda^4_3 +\lambda^1_2 \lambda^5_3 .
\end{array}
\end{equation*}
Using this, from the coefficients of $v^{13}$, $v^{14}$, $v^{16}$, $v^{23}$ and $v^{34}$ in the condition \eqref{condiciones} for $k=8$ we have the following equalities:
\begin{equation*}
\begin{split}
&\lambda^8_4= (\lambda^3_3 + \lambda^4_3) \lambda^5_1 -(2\,\lambda^2_1 +\lambda^4_1) \lambda^5_3 + \lambda^1_1 \lambda^6_3 +\lambda^2_1 \lambda^7_3 , \\[2pt]
&\lambda^8_5=(\lambda^3_3 + \lambda^4_3) \lambda^5_2 -(2\,\lambda^2_2 +\lambda^4_2) \lambda^5_3 + \lambda^1_2 \lambda^6_3 +\lambda^2_2 \lambda^7_3 , \\[2pt]
&\lambda^8_6=-2(\lambda^2_1)^2\lambda^3_3 -2\,\lambda^2_1 \lambda^3_3 \lambda^4_1 + \left( (\lambda^1_1)^2 +(\lambda^2_1)^2 \right)\lambda^4_3 +2\,\lambda^1_1 \lambda^3_3 \lambda^5_1 +2\,\lambda^1_1 \lambda^2_1 \lambda^5_3 , \\[2pt]
&\lambda^8_7=-\lambda^2_1(\lambda^3_3)^2 -\lambda^2_1 \lambda^3_3 \lambda^4_3 +\lambda^1_1 \lambda^3_3 \lambda^5_3 , \\[2pt]
&\lambda^8_8=(\lambda^1_1\lambda^2_2-\lambda^1_2\lambda^2_1) (\lambda^3_3)^2.
\end{split}
\end{equation*}
Now, 
considering these equalities together with the solutions \eqref{sol-p1-p2}, the conditions \eqref{condiciones} for all $1\leq k\leq 8$ are equivalent to the following system: 
\begin{equation}\label{sistema}
\begin{split}
& 
\lambda^4_1 -  \sqrt{3}\,i\,  \vartheta\,  \xi\,  \lambda^4_2 + 
     \sqrt{3}\,i\,  \vartheta\,  \lambda^5_1 - \xi\,  \lambda^5_2 = -4 \xi\,  \lambda^1_2,\\[4pt]
&
3 \lambda^4_1 -  \sqrt{3}\,i\,  \vartheta\, \xi\,  \lambda^4_2 + \sqrt{3}\,i\,  \vartheta\,  \lambda^5_1 - 3 \xi\,  \lambda^5_2 = -4 \xi\, \alpha\, (\lambda^1_2)^2,\\[4pt]
&
    3\sqrt{3}  \,i\,  \vartheta\,  \xi\,  \lambda^4_1 + 3 \lambda^4_2 + 3 \xi\,  \lambda^5_1 + 
    3 \sqrt{3}  \,i\,  \vartheta\,  \lambda^5_2 = -4 \lambda^1_2\,( \sqrt{3} \,i\,  \vartheta -2 \xi\,  \lambda^1_2) ,\\[4pt]
    &
\sqrt{3}\,i\,  \vartheta\,  \lambda^4_1 - 3 \xi\,  \lambda^4_2 + 3 \lambda^5_1 - 
    \sqrt{3}\,i\,  \vartheta\,  \xi\,  \lambda^5_2 = -4\,\xi\, \lambda^1_2\,( \alpha' -  \frac{i}{\sqrt{3}}\,  \vartheta\,  \alpha\, \lambda^1_2) ,\\[4pt]
  &
    3  \sqrt{3}\,i\,  \vartheta\,  \xi\,  \lambda^4_1 + 3 \xi\,  \lambda^5_1 
    = -4 \lambda^1_2\,( \sqrt{3}\,i\,  \vartheta + \xi\,  \lambda^1_2) ,\\[4pt]
 &
 4 \xi\,  \alpha\, (\lambda^1_2)^3 + 
    2  i\, \sqrt{3} \vartheta\,  \xi\,  \alpha\, (\lambda^1_2)^2 \lambda^4_3 - 
    2  i\, \sqrt{3} \vartheta\,  \lambda^1_2 \lambda^5_1 - 3 \lambda^4_2 \lambda^5_1 + 
    6 \xi\,  \lambda^1_2 \lambda^5_2  \\[4pt]
&    + 3 \lambda^4_1 \lambda^5_2 + 
    2 \xi\,  \alpha\, (\lambda^1_2)^2 \lambda^5_3 + 3 \lambda^1_2 \lambda^6_1 - 
     i\, \sqrt{3} \vartheta\,  \xi\,  \lambda^1_2 \lambda^6_2 +  i\, \sqrt{3} \vartheta\,  \lambda^1_2 \lambda^7_1 - 
    3 \xi\,  \lambda^1_2 \lambda^7_2 = 0.
\end{split}
\end{equation}
     
We focus our attention on the first five equations in~\eqref{sistema} and consider them as a system of five linear equations in $\lambda^4_1$, $\lambda^4_2$, $\lambda^5_1$ and $\lambda^5_2$. 
Applying Gaussian elimination, one reaches the matrix
 $$
\left(
\begin{array}{cccc|c}
1 & - \sqrt{3}\,i\,  \vartheta\,  \xi &  \sqrt{3}\,i\,  \vartheta & -\xi  
	& -4 \xi\,  \lambda^1_2   \\[3pt]
0 & 1 & -\xi & 0 
	& -\frac{2\,i}{\sqrt 3}\,\vartheta\,\lambda^1_2\,(3-\alpha\,\lambda^1_2) \\[3pt]
0 & 0 & 1 & \sqrt 3\,i\,\vartheta\,\xi 
	& \frac{2}{3}\,\vartheta\,\xi\,\lambda^1_2
	\big(
	2\,\vartheta\,\xi\,\lambda^1_2 - \sqrt 3\,i\,(1-\alpha\,\lambda^1_2)
	\big)\\[3pt]
0 & 0 & 0 & 0
	& -\frac{4}{\sqrt 3}\,\vartheta\,\xi\,\lambda^1_2\,
	\big(
	\sqrt 3\,\vartheta\,\alpha' + 2\,i\,(3-2\,\alpha\,\lambda^1_2)
	\big) \\[3pt]
0 & 0 & 0 & 0
	& -4\,\vartheta\,\lambda^1_2\,
	\big(
	2\,\vartheta\,\xi\,\lambda^1_2 + \sqrt 3\,i\,(2-\alpha\,\lambda^1_2)
	\big)
\end{array}\right).
$$
By hypothesis $\lambda^1_2\neq 0$, so the system can be solved if and only if
$$\begin{cases}
\sqrt 3\,\vartheta\,\alpha' + 2\,i\,(3-2\,\alpha\,\lambda^1_2)=0,\\
2\,\vartheta\,\xi\,\lambda^1_2 + \sqrt 3\,i\,(2-\alpha\,\lambda^1_2)=0.
\end{cases}$$
If $\alpha=0$, then $\alpha'$ would be an imaginary number, which is not possible. Otherwise, $\lambda^1_2$ can be solved from the first equation and substituting its value in the second one, we get
%$$2\,\alpha\,\sqrt 3\,i-2\sqrt 3\,\xi\,\alpha'\,i-3\,\vartheta\,\alpha\,\alpha'+12\,\xi\,\vartheta=0.$$
$$2\,\sqrt 3\,i\,\big(\alpha-\xi\,\alpha'\big)-3\,\vartheta\,\big(\alpha\,\alpha'-4\,\xi\big)=0.$$
Since $\alpha,\alpha'\in(0,\infty)$, the imaginary part of this equation gives $\alpha'=\alpha$.

Hence, we have proved that given $\alpha,\alpha'\in[0,+\infty)$, the algebras $\mathfrak \frg_{11}^{\alpha, 1}$ and $\mathfrak \frg_{11}^{\alpha', 1}$ are isomorphic over $\C$ if and only if $\alpha=\alpha'$.

\section*{Acknowledgments}
\noindent 
We wish to thank Mario Garc\' ia-Fern\'andez for very useful discussions about Bismut flat spaces that led us to the result in Proposition~\ref{counterex}. 
We also thank Dan Popovici and Jonas Stelzig for many helpful discussions on the Fr\"olicher sequence of complex manifolds and related issues.
This work has been partially supported by grant PID2020-115652GB-I00, funded by MCIN/AEI/10.13039/501100011033,
and by grant E22-17R ``Algebra y Geometr\'{\i}a'' (Gobierno de Arag\'on/FEDER).

\end{document}